\documentclass{amsart}
\usepackage{enumerate}
\usepackage{amsmath}
\usepackage{amsthm}
\usepackage{amsfonts}
\usepackage{amssymb}
\usepackage{mathrsfs}
\usepackage{stmaryrd}

\usepackage[pdfpagescrop={90 60 520 725}, bookmarks, colorlinks, breaklinks]{hyperref}  
\hypersetup{linkcolor=blue,citecolor=blue,filecolor=black,urlcolor=blue}

\newtheorem{theorem}{Theorem}[section]
\newtheorem{proposition}[theorem]{Proposition}
\newtheorem{conjecture}[theorem]{Conjecture}
\newtheorem{lemma}[theorem]{Lemma}
\newtheorem{corollary}[theorem]{Corollary}
\newtheorem{definition}[theorem]{Definition}
\theoremstyle{remark}

\newtheorem{remark}[theorem]{Remark}
\theoremstyle{plain}

\newcommand\wt\widetilde
\newcommand\sheaf\mathcal
\newcommand\complex\mathscr
\newcommand{\exterior}[1]{\ensuremath{{\textstyle\bigwedge}^{\! #1 }}}
\newcommand\lto\longrightarrow

\newcommand\poly{\ensuremath{H^*_{\sheaf E}(X)}}
\newcommand\polybeta{\ensuremath{H^*_{\sheaf E_\beta}(\M\beta)}}
\newcommand\polypq[2]{\ensuremath{H^{#2}(X, \exterior {#1} \sheaf E^\vee)}}

\newcommand\Z[2]{\ensuremath{\mathfrak  Z^{(#1)}_{#2}}}
\newcommand\ZK[1]{\ensuremath{\mathfrak  Z^{K}_{#1}}}
\newcommand\Ss[2]{\ensuremath{{S_{#2}^{(#1)}}}}

\newcommand\mori[1]{\ensuremath{{\text{NE}}(#1)}}
\newcommand\moriz[1]{\ensuremath{{\text{NE}}(#1)_\ZZ}}

\newcommand\Sym{\ensuremath{\text{Sym}}}
\newcommand\Hom{\ensuremath{\text{Hom}}}
\newcommand\Pic{\ensuremath{\text{Pic}}}

\newcommand\PP{\ensuremath{\mathbb P}}
\newcommand\CC{\ensuremath{\mathbb C}}

\newcommand\ZZ{\ensuremath{\mathbb Z}}
\newcommand\RR{\ensuremath{\mathbb R}}
\newcommand\X{{\ensuremath{X}}}
\newcommand\M[1]{\ensuremath{X_{#1}}}

\newcommand\homo{S}

\begin{document}
\bibliographystyle{hep}

\parindent=0cm
\parskip=0.5cm

\title[A Mathematical Theory of Quantum Sheaf Cohomology]{A Mathematical 
Theory of Quantum Sheaf Cohomology}
\author{Ron Donagi}
\address{Department of Mathematics, University of Pennsylvania, Philadelphia, PA 19104}
\email{donagi@math.upenn.edu}
\author{Josh Guffin}
\address{Department of Mathematics, University of Pennsylvania, Philadelphia, PA 19104}
\email{guffin@math.upenn.edu}
\author{Sheldon Katz}
\address{Department of Mathematics, University of Illinois, Urbana, IL  61801}
\email{katz@math.uiuc.edu}
\author{Eric Sharpe}
\address{
Physics Department, Virginia Tech, Blacksburg, VA  24061}
\email{ersharpe@vt.edu}

\date{\today}

\begin{abstract}
The purpose of this paper is to present a mathematical theory 
of the half-twisted $(0,2)$ gauged linear sigma model and its 
correlation functions
that agrees with and extends results from physics.
The theory is associated to a smooth projective toric variety $X$ 
and a deformation $\sheaf E$
of its tangent bundle $T_X$.  
It gives a quantum deformation of the cohomology ring of the exterior algebra of $\sheaf E^*$.
We prove that in the general case, the correlation functions are independent of `nonlinear' deformations. We derive quantum sheaf cohomology relations that correctly specialize to the ordinary quantum cohomology relations described by Batyrev in the special case $\sheaf E = T_X$.
\end{abstract}

\maketitle
\thispagestyle{empty}

\section{Introduction}

The gauged linear sigma model (GLSM) was introduced in \cite{Witten:1993yc}
as a quantum field theory that is closely related to the nonlinear sigma
model (NLSM), but 
easier to analyze for both (2,2) and (0,2) versions.
Quantum cohomology relations for the (2,2) GLSM were described
in \cite{Batyrev:1993} and elaborated on by \cite{Morrison:1994fr}.  
In this paper, we present a mathematical
theory for the (0,2) GLSM and derive analogous results.
To put this work in context, we give some background and motivation from 
physics before focusing on the mathematical formulation.  More details from 
the physics perspective are given in the paper \cite{Donagi:2011} by the
authors.  

The (0,2) NLSM is a physical theory associated to a
Calabi-Yau threefold $X$ and a vector bundle $\sheaf E$ on $X$
satisfying $c_1(\sheaf E)=0$ and the 
Green-Schwarz anomaly cancellation condition
$c_2(\sheaf E)=c_2(T_X)$.  The half-twisted
(0,2) NLSM (sometimes called the $A/2$ model) is a closely-related but
simpler theory that can be constructed in a much more general situation.

Consider a compact K\"ahler manifold $X$ 
with a holomorphic 
vector bundle $\sheaf E \to X$ 
satisfying the equalities
\begin{enumerate}[(i)]
   \item $c_1(\sheaf E) = c_1(T_X)$, \label{cond:c2}
   \item $c_2(\sheaf E) = c_2(T_X)$ \label{cond:det}
\end{enumerate}
in the cohomology of $X$.
As these conditions extend the usual Green-Schwarz anomaly
cancellation condition of heterotic string theory
we will call such a
bundle \emph{omalous} (i.e.\ not anomalous).
Given an omalous bundle $\sheaf E$ on  $X$, 
the half-twisted (0,2) NLSM can be defined as a physical
theory associated to maps from a
genus-zero Riemann surface $\Sigma$ to $X$, 
with fermions associated to
$\sheaf E$.  This quantum field theory possesses a ``quasi-topological
subsector''; a subalgebra of vertex operators conjectured to be independent of
the complex structure on $X$ and referred to  as the \emph{quantum
sheaf cohomology} of $\sheaf E \to X$ \cite{Adams:2003d,Adams:2005tc}.   The
operators
in the quasi-topological sector are in one-to-one correspondence with
the sheaf cohomology $\oplus_{p,q}\polypq p q$.

Ignoring quantum corrections, the product of operators corresponds to the
cup product of corresponding classes
\[
H^q(X, \exterior p \sheaf E^\vee) \otimes
H^{q'}(X, \exterior {p'} \sheaf E^\vee)  \mathop{\longrightarrow}^\cup
H^{q+q'}(X, \exterior {p+p'} \sheaf E^\vee),
\]
where the cup product $\cup$ refers to the usual cup product on cohomology
followed by the product in the exterior algebra of $\sheaf E^\vee$.
In the full quantum theory, instanton corrections 
modify the product, but  for now
we discuss the classical algebra, which we call the \emph{polymology} of
$\sheaf E$.  

\begin{definition}
  The \emph{polymology} of a vector bundle
   $\sheaf E$ is the associative algebra
   \begin{equation}
      \poly := \bigoplus_{p,q} \polypq p q
      \label{eq:classicalalgebra}
   \end{equation}
   equipped with the cup product.
\end{definition}
If 
$\sheaf E=T_X$, i.e.\ if the (0,2) 
theory is actually a (2,2) theory, then the polymology
is canonically isomorphic to the ordinary cohomology of $X$ by Hodge theory.

The polymology can be defined for any vector bundle, but
if $\sheaf E$ is omalous, a choice of isomorphism $\det
\sheaf E^\vee \simeq \omega_X$ induces an isomorphism $\psi:H^n(X,
\det \sheaf E^\vee) \simeq H^n(X, \omega_X)$, where $n$ is the dimension
of $X$. This isomorphism in turn induces a pairing
\begin{equation}
   (\alpha, \beta) = \int_X \psi(\alpha\cup\beta)
   \label{eq:polyPairing}
\end{equation}
satisfying $(\alpha\cup\beta, \gamma) = (\alpha, \beta \cup
\gamma)$, which is perfect by Serre duality:
\begin{equation}
   \begin{split}
      H^p(X, \exterior q \sheaf E^\vee)^\vee
      &\simeq H^{n-p}(X, \exterior q\sheaf E \otimes \omega_X)\\
      &\simeq H^{n-p}(X, \exterior q\sheaf E \otimes \exterior n\sheaf E^\vee)\\
      &\simeq H^{n-p}(X, \exterior {n-q}\sheaf E^\vee).\\
   \end{split}
   \label{eq:serreDuality}
\end{equation}
It follows easily that the polymology of an omalous vector bundle admits
the structure of a bigraded Frobenius algebra.

The classical correlation functions can be identified with the
pairing \eqref{eq:polyPairing}.
Note that as $H^n(X, \exterior n\sheaf E^\vee)$ is 
isomorphic to the complex numbers, any such isomorphism defines the trace in 
the Frobenius algebra structure.  However, this isomorphism is not canonical.  
We deal with this normalization issue by simply defining the classical
correlation functions to live in the one-dimensional vector space 
$H^n(X, \exterior n \sheaf E^\vee)$.  

Before discussing the \emph{quantum sheaf cohomology} of $(X,\sheaf
E)$, a brief discussion is in order about two relevant quantum field
theories: the (2,2) NLSM and the (2,2) GLSM.  A little later we will
discuss the (0,2) GLSM and the (0,2) NLSM.

Given a smooth projective variety $X$, the quantum corrections to any
of these quantum field theories can be computed perturbatively using a
compactification of the space of holomorphic maps
$f:\PP^1\to X$ with $f_*[\PP^1]=\beta$, for each $\beta\in
H_2(X,\ZZ)$, and performing an integration over this compactification.  
The (2,2) NLSM is well-understood mathematically as
ordinary quantum cohomology, with the appropriate compactification
being the moduli space $\overline{M}_{0,3}(X,\beta)$ of genus~0 stable
maps of class $\beta$.  If $X$ is a toric variety, there is the 
linear sigma model moduli space $X_\beta$, which is a toric
compactification used for the (2,2) GLSM
\cite{Witten:1993yc,Morrison:1994fr}, leading to a quantum cohomology
ring whose structure was described by Batyrev in \cite{Batyrev:1993}.
Either of these quantum cohomology
rings are deformations of the ordinary cohomology ring $H^*(X)$.
A comparison between Batyrev's quantum cohomology ring and the
ordinary quantum cohomology ring follows from \cite{givental:1998} and
is described in \cite{Cox:2000vi}.  These
two cohomology rings become identified after a change in variables 
(the \emph{mirror map}).  The Batyrev quantum cohomology ring is identical
to the usual quantum cohomology ring if $X$ is Fano.

The quantum sheaf cohomology described in this paper arises from the
half-twisted (0,2) GLSM, and extends Batyrev's quantum cohomology
ring.  We will also explain how the identical moduli space $X_\beta$ used in 
the (2,2) GLSM moduli space can also be used to describe the (0,2) GLSM,
independent of $\sheaf E$.

Quantum sheaf cohomology
is a quantum deformation of the Frobenius algebra
structure on the polymology of $(X,\sheaf E)$.  This is analogous to either 
of the two versions of the quantum cohomology of $X$
mentioned above.

Physics tells us from the half-twisted (0,2) NLSM that a quantum sheaf
cohomology ring is associated to $(X,\sheaf E)$ for any omalous vector
bundle $\sheaf E$ on $X$.  Unfortunately, a mathematically-precise
version of such a theory does not yet exist.  However, physical
arguments providing an approach to such a mathematical version are
given in \cite{Katz:2004nn}.  Furthermore, one can speculate that the
relevant compactification of the space of maps will be
$\overline{M}_{0,3}(X,\beta)$ as in the (2,2) NLSM, independent of the
choice of bundle $\sheaf E$.  One can also speculate that the GLSM and
NLSM versions of quantum sheaf cohomology will be identified by a
change of variables analogous to the mirror map.

We now describe the ingredients of 
the half-twisted (0,2) GLSM.  Although such theories
are more general, we will only describe them in the situation we
consider.  Let $X$ be a smooth projective toric variety, and let $\sheaf
E$ be a deformation of the tangent bundle $T_X$ arising from a
deformation of the toric Euler sequence (to be described in
(\ref{eq:Tpresentation}) below).  The half-twisted (0,2) GLSM is associated to
$(X,\sheaf E)$, and has a \emph{quasi-topological sector}
whose operators are generated by the symmetric
algebra on $H^2(X,\CC)$ and become isomorphic to $\poly$ in the
classical limit.  
In physics, the quasi-topological sector arises as the set of operators lying 
in the
kernel of the scalar supercharge whose holomorphic conformal weight
vanishes.  

One of our main results is the calculation of 
the classical polymology of $(X,\sheaf E)$ in
Theorem~\ref{thm:polymologyIsomorphism}.  The algebra $\poly$ is 
naturally a quotient
of this symmetric algebra, relating the GLSM to the NLSM.  
To each primitive collection
in the toric variety $X$ is associated a generator of the Stanley-Reisner
ideal $\mathrm{SR}(X,\sheaf E)$.  The classical polymology is the 
quotient of this symmetric algebra by $\mathrm{SR}(X,\sheaf E)$.

Since $X_\beta$ is itself
toric, the same result can be used to compute quantum corrections in all 
instanton sectors.  We introduce a direct system of polymologies in 
Section~\ref{sec:direct} which allows us to sum over the
instanton sectors and rigorously define the quantum sheaf cohomology ring
by abstracting the physical
notion of an operator.  For each primitive collection, a quantum deformation
of the corresponding Stanley-Reisner generators can be written down 
(\ref{eq:qcrelation}), as proposed in the physics literature.  
Our final main result is:

\noindent
{\bf Theorem \ref{thm:qcrelations}.}
The quantum cohomology relations (\ref{eq:qcrelation})
hold for all primitive collections $K$.

Here is an outline of the rest of the paper.

We begin Section~\ref{sec:toric} by recalling standard concepts and
notation from toric geometry.  Then we recall the toric Euler sequence
of a smooth toric variety $\X$, which gives a presentation of its
tangent bundle.  Deformations of this sequence are
presentations of deformations $\sheaf E$ of the tangent bundle, which
complete the input data needed to define the $(0,2)$ GLSM.  We then
introduce a generalized Koszul complex that plays a fundamental role
in our analysis.  We conclude the section with the computation of the
sheaf cohomology of certain $T$-invariant divisors, which enable us to
chase through exact cohomology sequences associated with the
generalized Koszul complex.

In Section~\ref{sec:polymology} we compute the polymology of
$(X,\sheaf E)$.  Let $W=H^2(\X,\CC)$ and let $K\subset \Sigma(1)$ be a
primitive collection, where $\Sigma(1)$ as usual denotes the edges of
the fan for $\X$.  By a diagram chase, we associate to each $K$ an element
$Q_K\in\mathrm{Sym}^kW$, where $k=|K|$.  We then define the
Stanley-Reisner ideal $\mathrm{SR}(X,\sheaf E)$ of $\sheaf E$ to be
the ideal in $\mathrm{Sym}^*W$ generated by the $Q_K$.  The main
result of Section~\ref{sec:polymology} is that the polymology of
$\sheaf E$ is isomorphic to the quotient of $\mathrm{Sym}^*W$ by the
Stanley-Reisner ideal of $\sheaf E$
(Theorem~\ref{thm:polymologyIsomorphism}).

In Section~\ref{sec:quantum} we describe the GLSM moduli space
$\X_\beta$ associated to an effective class $\beta\in H_2(\X,\ZZ)$ and
an induced vector bundle $\sheaf E_\beta$ on $\X_\beta$.  The
correlation functions are defined in terms of the polymology of
$(\X_\beta,\sheaf E_\beta)$.  We show that the polymology of
$(\X_\beta,\sheaf E_\beta)$ is a quotient of $\mathrm{Sym}^*W$ by an
ideal generated by powers of the $Q_K$ given by (\ref{eq:qbk}).  We
then introduce a direct system of polymologies that will allow us to
compare correlation functions for different $\beta$, and then define
the correlation functions after introducing {\em four-fermi terms\/}
in (\ref{eq:fourfermi}) that play a role for the $(0,2)$ GLSM similar
to that of the virtual fundamental class of Gromov-Witten theory.
Finally, we define the quantum sheaf cohomology ring abstractly and
then compute the quantum cohomology relations in
Theorem~\ref{thm:qcrelations}.  In \cite{McOrist:2008ji}, 
predictions were made for the image of the relations in a
localization of the ring (following a standard procedure in the
physics literature).  It is straightforward to verify that our relations 
descend to the relations of
\cite{McOrist:2008ji} in that localization.  See \cite{Donagi:2011} for 
details.

We now turn to toric geometry to elaborate on the set-up.

\section{The toric setting}
\label{sec:toric}

We start by fixing some notation associated to toric varieties.  A good
general reference for toric varieties is \cite{Cox:2010tv}.

Let $X = X_\Sigma$ be a smooth projective toric variety 
of dimension $n$ with fan
$\Sigma$ whose support lies in $N_\RR\simeq \RR^n$, where $N$ is the lattice
of one-parameter subgroups of the dense torus of $X$.  We will denote by
$M$ the lattice of characters dual to $N$, by $\Sigma(1)$ the set of
one-dimensional cones of the fan, and we will write $\bigoplus_\rho$ and
$\sum_\rho$ in place of $\bigoplus_{\rho \in \Sigma(1)}$ and $\sum_{\rho
\in \Sigma(1)}$, respectively.  To each $\rho \in \Sigma(1)$ is
associated a torus-invariant Weil divisor denoted $D_\rho$, a unique
minimal generator $v_\rho$ of the semigroup $N \cap \rho$, and a
canonical section $x_\rho \in H^0(\X, \sheaf O_\X(D_\rho))$.  These
canonical sections freely generate the homogeneous coordinate ring of $\X$:
\begin{equation}
   S:=\CC[x_\rho \; |\; \rho \in \Sigma(1)].
\end{equation}
The homogenous coordinate ring $S$ has a natural grading by $\Pic(\X)$, 
which assigns to $x_\rho$ the degree $[D_\rho]\in \Pic(\X)$.  

For each $T$-invariant Weil divisor $D= \sum_\rho a_\rho D_\rho$ we have a natural isomorphism 
\begin{equation}
S_{[D]} \simeq H^0(\X, \sheaf O_\X(D)),
\label{eq:homopiece}
\end{equation}  
where as usual, $S_{[D]}$ denotes the graded piece of $S$ of degree $[D]$.
To describe this isomorphism, we
associate to $D$ the polytope 
\[
\Delta_D = \{ m \in M_\RR \; | \; \langle m, v_\rho \rangle \geq - a_\rho
\text{ for all } \rho \in \Sigma(1)\}.
\]
Then the isomorphism (\ref{eq:homopiece}) is conveniently described by the 
identification of basis elements
\begin{equation}
\prod_\rho x_\rho^{\langle m, v_\rho\rangle + a_\rho}\leftrightarrow
\chi^m,
\label{eq:isobasis}
\end{equation}
where $m$ ranges over $\Delta_{D}\cap M$ and $\chi^m$ is the 
character associated
to $m\in M$, thought of as a meromorphic
function with at worst poles on $D$.
Note in particular that for each $\rho$ the trivial character $\chi^0$ (i.e.\ the constant
function 1) is the section of $\sheaf
O_\X(D_\rho)$ associated to 
$x_\rho$ via the isomorphism (\ref{eq:homopiece}) for $D=D_\rho$.

Since $\X$ is smooth and toric, the class group of Weil divisors, the
Picard group, and the integral cohomology are all isomorphic.  
We associate to each 
$m\in M$ the element $e_m\in\ZZ^{\Sigma(1)}$ defined by $e_m(\rho)=
\langle m,v_\rho\rangle$.
Then
$\Pic(\X)$ admits a presentation as
\begin{equation}
0 \to M \to \ZZ^{\Sigma(1)} \to \Pic(\X) \to 0,
\label{eq:pic}
\end{equation}
where the first non-trivial morphism is $m \mapsto e_m$ and the basis element of
$\ZZ^{\Sigma(1)}$ dual to $\rho$ maps to $[D_\rho]$.

Let $W=H^2(X,\CC)$.  Then the tangent bundle of 
$\X$ fits into a short exact sequence 
known as the toric Euler sequence:
\begin{equation}
   0 \longrightarrow \sheaf O_\X \otimes_\CC W^\vee 
\mathop{\longrightarrow}^{E_0}
   {\bigoplus}_\rho \mathcal
   O_\X(D_\rho) \longrightarrow T_\X \longrightarrow 0.
   \label{eq:Tpresentation}
\end{equation}
Thinking of $E_0$ as an element of $\bigoplus_\rho S_{[D_\rho]} \otimes W$,
the $\rho^\text{th}$ component of $E_0$ is $x_\rho \otimes [D_\rho]$.

Recall that a collection of edges $K \subset \Sigma(1)$ is a \emph{primitive
collection} if $K$ does not span any cone in $\Sigma$, but every proper
subcollection of $K$ does.  Equivalently, the intersection of the divisors
$D_\rho$ with $\rho\in K$ is empty, but the intersection of any proper subset
of these divisors is nonempty.
Following the presentation of
\cite[\S 8.1.2]{Cox:2000vi},  we
define two ideals in the homogeneous coordinate ring $\homo$: 
\begin{equation}
   \begin{split}
      P(\X)  & = \left ( \sum_{\rho} \langle m,
      v_\rho\rangle x_\rho \; \Big| \; m \in M\right ) \\
      SR(\X) & = \left ( \prod_{\rho \in K} x_\rho \;\Big|\; K \text{ a primitive collection of } \Sigma\right).
   \end{split}
   \label{eq:stanleyReisnerLinearEquivalence}
\end{equation}
The former is the ideal of linear equivalences, so that
\begin{equation}
\mathrm{Sym}(W) \simeq \homo /P(\X).
\label{eq:swlineq}
\end{equation}
The second ideal in (\ref{eq:stanleyReisnerLinearEquivalence}) is
known as the \emph{Stanley-Reisner} ideal of $\X$. 
It is well
known \cite{Fulton:1993tv,Oda:1988tv} that there is
an isomorphism of $\ZZ$-graded algebras
\begin{equation}
{ \homo }/\left({P(\X) + SR(\X)}\right)\simeq H^*(\X, \CC)
\label{eq:toriccohomology}
\end{equation}
induced by sending a generator $x_\rho$ of $S$ to $[D_\rho]$.

For later use, we recall the description of toric varieties as quotients.
There is a natural action of $G=\Hom(\Pic(X), \CC^*)$ on $\CC^{\Sigma(1)}$ 
following from the inclusion $G\subset \Hom(\ZZ^{\Sigma(1)},\CC^*)$ 
derived from (\ref{eq:pic}).  For each $\sigma\in\Sigma$ define
\begin{equation}
x^{\hat{\sigma}}=\prod_{\rho\not\in\sigma(1)}x_\rho
\label{eq:xsigmahat}
\end{equation}
and the \emph{irrelevant ideal} 
\begin{equation}
B(\Sigma)=\left( x^{\hat{\sigma}}\mid \sigma\in\Sigma\right).
\label{eq:irrelevant}
\end{equation}
Thinking of the $x_\rho$ as coordinate 
functions on $\CC^{\Sigma(1)}$, we define the subset
$Z(\Sigma)\subset \CC^{\Sigma(1)}$ as the vanishing locus of the
irrelevant ideal.  Then $Z(\Sigma)$ is $G$-invariant and
\begin{equation}
\X = \left(\CC^{\Sigma(1)} - Z(\Sigma)\right)/{G}.
\label{eq:quotientdescription}
\end{equation}
For later use, we note that it is well known that $Z(\Sigma)$ can be described in terms of
primitive collections.  If $K=\{\rho_1,\ldots,\rho_k\}$ is a primitive 
collection, let $L_K\subset \CC^{\Sigma(1)}$ be the linear subspace defined
by $x_{\rho_1}=\ldots=x_{\rho_k}=0$.  Then
\begin{equation}
Z(\Sigma)=\cup_K L_K,
\label{eq:zspc}
\end{equation}
where the union is taken over all primitive collections $K$.  The fan $\Sigma$
can also be recovered from the set of primitive collections as the set of cones
spanned by collections of edges that do not contain any primitive collection.  See 
\cite{Cox:2010tv} for example.

It will also be useful to note that $\Pic(X)$ can be recovered from $G$ as
\begin{equation}
\Pic(X)\simeq\Hom(G,\CC^*),
\label{eq:picg}
\end{equation}
by duality for finitely-generated abelian groups.

\subsection{Toric deformations of the tangent bundle.}

To define a half-twisted (0,2) GLSM, we need a presentation
of a vector bundle $\sheaf E$ obtained from 
(\ref{eq:Tpresentation}) by simply changing the map $E_0$:
\begin{equation}
   0 \longrightarrow \sheaf O_\X \otimes_\CC W^\vee 
\mathop{\longrightarrow}^{E}
   {\bigoplus}_\rho \mathcal
   O_\X(D_\rho) \longrightarrow \sheaf E \longrightarrow 0.
   \label{eq:Epresentation}
\end{equation}
Both the bundle and the presentation are required to define the GLSM.  
We will sometimes abuse terminology slightly by referring to the bundle 
$\sheaf E$ as a 
\emph{toric deformation of the tangent bundle\/}, but we always have a fixed
presentation (\ref{eq:Epresentation}) in mind.
Specifying a map $E$ is not sufficient; it is required
that the cokernel $\sheaf E$ of $E$ is locally free, or equivalently that
\begin{equation}
E^t:\oplus_\rho\sheaf O(-D_\rho)\to W\otimes\sheaf O_{\X}
\label{eq:elocfree}
\end{equation}
is surjective.

As with $E_0$, 
the map $E$ can be viewed as a section of $\oplus_\rho H^0(X,\sheaf O(D_\rho))
\otimes W=\oplus_\rho S_{[D_\rho]}\otimes W$.

The components $E_\rho$ of $E$ can be thought of as 
$W$-valued sections of $\sheaf O(D_\rho)$.
We will sometimes express these sections as
\[
E_\rho=\sum_{m\in\Delta_{D_\rho}\cap M}a_{\rho m}\chi^m,
\]
where $a_{\rho m}\in W$, or as
\[
E_\rho=\sum_{m\in\Delta_{D_\rho}\cap M}a_{\rho m}x_\rho\prod_{\rho'} 
x_{\rho'}^{\langle m,v_\rho'\rangle}
\]
using the identification (\ref{eq:isobasis}).

The terms $a_{\rho m}\chi^m$ with $x_\rho\prod_{\rho'} x_{\rho'}^{\langle
m,v_\rho'\rangle}$ a linear monomial in the homogeneous coordinates will 
be called a \emph{linear term}; the other terms will be called the 
\emph{nonlinear terms}.
A toric deformation of the tangent bundle containing only linear terms
will be called a \emph{linear deformation}.
Linear
deformations play a
significant role in physicists' analyses of quantum sheaf cohomology
\cite{Katz:2004nn, Guffin:2007mp, McOrist:2007kp, McOrist:2008ji}
in localized rings, and
we will see that they capture the essence of the polymology associated
to a toric deformation of the tangent bundle.

Throughout, we will make extensive use of a generalized Koszul 
complex associated to deformations of the toric Euler sequence. 
In order to simplify notation, put $Z=\oplus_\rho\sheaf O(-D_\rho)$ and
then for $0 \leq k$ and $0 \leq j \leq k$ define 
\[
\Z k j  := \exterior j Z \otimes \Sym^{k-j} W.
\]
Let $\sheaf E$ be a 
deformation of the tangent bundle in the preceding sense.  
The dual of the exact sequence
\eqref{eq:Epresentation} induces an injection 
$\exterior k \mathcal E^\vee \to \exterior k Z$
and maps $\alpha_j:\Z k j \to \Z k {j-1}$ defined as
\begin{equation}
   \alpha_j:
   (z_1 \wedge \cdots \wedge z_j)\otimes s \mapsto \sum_{\ell=1}^j (-1)^{\ell-1}
   (z_1 \wedge \cdots \wedge \hat z_\ell\wedge \cdots  \wedge
   z_j)\otimes \big[E^\vee(z_\ell) \odot s\big],
   \label{eq:koszulMaps}
\end{equation}
where $E$ is the injection in 
\eqref{eq:Epresentation} and $\odot$ is multiplication in $\Sym^*W$.
These maps may be arranged into an exact sequence
\begin{equation}
   0 \lto \exterior k \sheaf E^\vee \lto \Z k k \lto \Z k {k-1} \lto \cdots \lto \Z k 1 \lto \Z k 0 \lto 0.
   \label{eq:koszulType}
\end{equation}
Exactness follows since the maps in (\ref{eq:koszulType}) are natural and 
the analogous complex formed from a short exact sequence of vector spaces
is easily seen to be exact.

\subsection{A vanishing result for certain toric line bundles}

We will make extensive use of the line bundles
$\sheaf O_\X(D)$ associated to T-invariant divisors $D$ on $\X$,
and in particular
Theorem~\ref{thm:simp} below, which appears in
\cite{Demazure:1970} and is reproduced here for convenience.

Consider a Weil divisor $D = \sum_\rho a_\rho D_\rho$ and define
\[ \Sigma_{D,m} = \{\rho \in \Sigma(1) \; | \; \langle m, v_\rho \rangle < -a_\rho\}.\]
Then let $V_{D,m}$\footnote{We are following the notation of
\cite{Cox:2010tv}, adapting it slightly; our $V_{D,m}$ matches their 
$V_{D,m}^\text{simp}$.} 
be the union of polytopes in $N_\mathbb R$
\[
V_{D,m} = \mathop{\bigcup_{\sigma \in \Sigma}}_{\sigma(1) \subset
\Sigma_{D,m}} \text{Conv}(v_\rho \; |\;  \rho \in \sigma).
\]
Since $\sheaf O_\X(-D)$
is a torus-equivariant bundle, $H^j(\X,\sheaf O_\X(-D))$ decomposes as a 
direct sum of 
weight spaces $H^j(\X,\sheaf O_\X(-D))_m$ with $m \in M$.

\begin{theorem}[Proposition~6 of \cite{Demazure:1970}]
   \label{thm:simp}
   Let $D = \sum_\rho a_\rho D_\rho$ be a $T$-invariant 
Weil divisor on $\X$.  For
   $m\in M$ and $p\geq 0$,
   \[H^p(\X, \mathcal O_\X(D))_m \simeq \widetilde
   H^{p-1}(V_{D,m},\mathbb C).\]
\end{theorem}
Here $\widetilde H$ denotes the reduced cohomology.
Consider a subset $K \subset \Sigma(1)$ and set $D_K = \sum_{\rho \in K}
D_\rho$.

\begin{lemma}
   For all j and all $K \subset \Sigma(1)$,
   \(H^j(\X,\sheaf O_\X(-D_K)) = H^j(\X,\sheaf O_\X(-D_K))_0\).
   That is, the cohomology of $\sheaf O_\X(-D_K)$ is purely of weight 0.
   \label{lem:invariantCohomology}
\end{lemma}

\begin{proof}
   By Theorem \ref{thm:simp}, $H^j(\X,\sheaf  O_\X(-D_K))_m$ is the reduced 
cohomology
   of the topological space $V_{-D_K,m}$ obtained as follows: $\Sigma$
   determines a simplicial complex whose faces are $\text{Conv}(v_\rho | \rho
   \in \sigma(1))$ for $\sigma \in \Sigma$.  $V_{-D_K,m}$ is the subcomplex
   corresponding to those $\sigma$ such that $\sigma(1)$ is contained in: 
   \[
   \Sigma_{-D_K,m} := \{\rho \in K | \langle m,v_\rho\rangle \leq 0\} \cup
   \{\rho \notin K | \langle m,v_\rho\rangle < 0\}.
   \]
   Here we use the fact that the coefficients of the $D_\rho$ in $-D_K$
   are either 0 or $-1$.  The set $\Sigma_{-D_K,m}$  is invariant under
   rescaling $m$ by a positive integer, and therefore so are $V_{D,m}$
   and $H^j(\X,\mathcal  O_\X(-D_K))_m$. If the latter were
   non-vanishing for some non-zero $m$, $H^j(\X,\sheaf  O_\X(-D_K))$ would not
   be finite dimensional, contradicting the projectivity of $X$.
\end{proof}

\begin{proposition}
   Let $\Sigma$ be a simplicial fan and $K\subset \Sigma(1)$.  Setting  $k =
   |K|$  and $D_K = \sum_{\rho\in K} D_\rho$ as before,  we have that 
   \begin{enumerate}[i)]
      \item For all $\ell \geq k$, $H^\ell (\X, \mathcal O_\X(-D_K)) = 0$
      \item If $\cap_{\rho\in K} D_\rho \ne \emptyset$, then for all
         $\ell\in \ZZ$,  $H^\ell(\X, \mathcal
         O_\X(-D_K)) = 0$.
      \item If $K$ is a primitive collection,
         \[
         H^\ell(\X, \sheaf O_\X(-D_K)) \simeq
         \begin{cases}
            \CC & \ell = k-1\\
            0 & \text{otherwise}
         \end{cases}
         \]
       \item If $K$ is not a primitive collection, then 
             $H^{k-1}(X,\sheaf O_X(-D_K))=0$.

   \end{enumerate}
   \label{prop:vanishing}
\end{proposition}

\begin{proof}
   We use Theorem \ref{thm:simp} and the notation therein throughout.  By Lemma \ref{lem:invariantCohomology}, we need
   only consider the torus-invariant part of the cohomology.  The relevant
   set of one-cones is $\Sigma_{-D_K,0} = K$. 

   \begin{enumerate}[i)]
      \item $V_{-D_K,0}$ is contained in the convex hull of $k$
         points and so can never contain a non-contractible $k-1$ cycle.
         Similarly, it does not contain any $\ell$ cycles with $\ell > k-1$.
      \item
         If the intersection is nonempty, $\text{cone}\{v_\rho\;|\; \rho \in 
K\} \in
         \Sigma$ and the $v_\rho$ are linearly independent since $\Sigma$ is
         simplicial: thus $V_{-D_K,0}$ is a $k-1$ simplex.
      \item
         Consider a primitive collection $K$.
         Since $K$ is primitive, every proper subset of $K$ spans a cone in 
$\Sigma$, so the
         simplicial complex takes the form
         \[
         V_{-D_K,0} = \bigcup_{\rho' \in K} \text{Conv}(v_\rho \; |\;  \rho \in K, \rho \ne \rho').
         \]
         This set is precisely the boundary of the $(k-1)$-simplex Conv$(v_\rho\;|\;
         \rho \in K)$, so that $V_{-D_K,0}$ is homeomorphic to the $(k-2)$ sphere
and the last claim follows immediately.
     \item If $K$ is not a primitive collection, we need only consider the
situation where $\cap_{\rho\in K}D_\rho=\emptyset$.  Then by comparison to the
analysis in iii) above, we see that either $V_{-D_K,0}$ has dimension
strictly less than $k-2$, or it has dimension $k-2$ and is homeomorphic
to a proper subcomplex of the above simplicial triangulation of $S^{k-2}$.
Either way we conclude that $\widetilde{H}^{k-2}
(V_{-D_K,0})=0$ and we are done.
   \end{enumerate}
\end{proof}

\begin{remark}
   An immediate consequence of $ii)$ is that for all $\ell\in \ZZ$ and
   $\rho \in \Sigma(1)$, $H^\ell(\X, \sheaf O_\X(-D_\rho)) = 0$.
\end{remark}

\section{$\poly$ for toric deformations of the tangent bundle}
\label{sec:polymology}

In this section, we study the algebra $\poly$, showing that as a
bigraded vector space it is isomorphic to $H^*(\X, \CC)$.  Multiplicatively 
it is
generated under the cup product by elements of $H^1(\X, \sheaf E^\vee)$.
We show that the relations among the generators may be given explicitly by
defining an ideal analogous to $SR(X)$.  Some of the results in this 
section are not used elsewhere in this paper, but we include them since
they could be useful in applications to the NLSM.

\subsection{Graded components}

We begin our study of $\poly$ by elucidating its vector space structure.
In particular, we show that it is diagonal, in the sense that its graded
components $\polypq pq$ vanish unless $p=q$.

\begin{proposition}
   \label{prop:gradedIsomorphism}
   Let $\sheaf E$ be a locally-free toric Euler sequence deformation.  Then
   for any $p$ and $q\ne p$, 
   \[
   \polypq pq = 0.
   \]
\end{proposition}

\begin{proof}
   For $q = 0$, the proposition holds trivially, since 
$\exterior p\sheaf E^\vee\subset \exterior p Z$ and $H^0(X, \exterior p Z)=0$.  
Thus, assume that $q > 0$.

   Consider the exact sequence in \eqref{eq:koszulType}. For each $1\leq j \leq p-1$ define
   \begin{equation}
      \Ss p j = \ker \Big(\Z p j \to \Z p {j-1}\Big),
      \label{eq:koszulSplitting}
   \end{equation}
   and set $\Ss p 0 := \Sym^p W \otimes \sheaf O_\X$ and $\Ss p p := \exterior p \sheaf E^\vee$. Then, for each $0 < \ell \leq
   p$,  we have a short exact sequence
   \begin{equation}
      0 \lto \Ss p \ell \lto \Z p \ell \lto \Ss p {\ell-1} \lto 0. 
      \label{eq:endSequences}
   \end{equation}

   We first show that for $q > p$, $\polypq pq = 0$.  
   Consider the 
   long exact sequence in cohomology induced by \eqref{eq:endSequences}:
   \[
   \cdots \to H^{q-1}(\X, \Z p \ell) \to H^{q-1}(\X, \Ss p {\ell -1}) \to H^q (\X, \Ss p \ell) \to H^q (\X, \Z p \ell) \to \cdots.
   \]
   By Proposition \ref{prop:vanishing}, we have that for $q \ge \ell$, 
$H^q(\X, \Z p \ell) = 0$, so that for $q > \ell$ we have
   \[
   H^{q-1}(\X, \Ss p {\ell -1}) \simeq H^q (\X, \Ss p \ell).
   \]
   By varying $\ell$, we obtain a chain of isomorphisms 
   \[
   H^{q-p}(\X, \Ss p 0) \stackrel \sim \longrightarrow
   H^{q-p+1}(\X,  \Ss p 1 ) \stackrel \sim \longrightarrow
   \cdots \stackrel \sim \longrightarrow
   H^q(\X, \Ss p p),
   \]
   which shows $\polypq pq\simeq 
 \Sym^p W \otimes H^{q-p}(\X, \sheaf O_\X) = 0$.

   Now, assume that $q < p$. By Serre duality 
$\polypq pq$ is dual to $\polypq {n-p} {n-q}$,
   and $n -q > n-p$, so the latter vanishes by the above considerations.
\end{proof}

\begin{corollary}
   Let $\sheaf E$ be a toric deformation of the tangent bundle.  
Then there is a
   canonical isomorphism
   \begin{equation}
      \poly \simeq \bigoplus_k \polypq k k.
      \label{eq:polyDiagonal}
   \end{equation}
   \label{cor:polyDiagonal}
\end{corollary}

\begin{remark}
   If we examine the particular case of $k = 1$, $\ell = 1$, we find
   that 
   \begin{equation}
      H^1(X,\sheaf E^\vee) \simeq H^0(\X, \sheaf O_\X \otimes W) \simeq W
      \label{eq:generators}
   \end{equation}
\end{remark}
by dualizing (\ref{eq:Epresentation}) to obtain 
\begin{equation}
0\to\sheaf E^\vee\to Z\to W\otimes \sheaf O \to 0
\label{eq:edkernel}
\end{equation}
and using $H^0(X,Z)=H^1(X,Z)=0$.

\subsection{Generators}

Now that we have a better idea of the linear structure of the algebra,
we would like to identify a set of minimal generators.  Corollary
\ref{cor:polyDiagonal} and the fact that the cohomology of smooth
projective toric varieties are generated by elements of $H^1(\X,
\Omega_\X^1)$ lead us to suspect that elements of $H^1(X,\sheaf E^\vee)
\simeq W$ generate $\poly$.   

As multiplication is defined using the cup product and the algebra is
diagonal, it is in fact commutative.  For $1 \leq k \leq n$,
the cup product of $k$ elements of $H^1(X,\sheaf E^\vee)$ is a linear map 
$\Sym^k  H^1(X,\sheaf E^\vee) \to \polypq kk$ that we will rewrite as
\begin{equation}
   m_k:\Sym^k  W \to \polypq kk.
   \label{eq:kMultiplication}
\end{equation}

\begin{lemma}
The map $m_k$ can be identified with
the extension class in $\mathrm{Ext}^k(\mathrm{Sym}^kW
\otimes \sheaf O_X,\exterior k\sheaf E^\vee)\simeq
\Hom(\mathrm{Sym}^kW,\polypq kk)$ associated to the generalized
Koszul complex
\eqref{eq:koszulType}.  
\end{lemma}

\begin{proof}
We sketch the argument.  First, $m_1$ is identified with the extension class 
$e\in\mathrm{Ext}^1(W\otimes \sheaf O,\sheaf E^\vee)$ of
(\ref{eq:edkernel}), an exact sequence that can be reinterpreted as
a quasi-isomorphism of $\sheaf E^\vee$ with the complex
\[
C^\bullet:\qquad 0\to Z\to W\otimes \sheaf O\to 0
\]
with $Z$ in degree~0. Thus, $e$ can be represented by the natural morphism of complexes $f:W\otimes \sheaf O\to C^\bullet[1]$
whose only nontrivial component is the identity map on $W\otimes \sheaf O$.  We can then tensor $f$ with itself $k$ times to
get a morphism 
\[
f^{\otimes k}:W^{\otimes k}\otimes\sheaf O\to 
\left(C^\bullet\right)^{\otimes k}[k].
\]
However, $({C^\bullet})^{\otimes k}$ 
is quasiisomorphic to $(\sheaf E^\vee)^{\otimes k}$, a statement
that we can rewrite as an exact sequence
\begin{equation}
0\to \left(\sheaf E^\vee\right)^{\otimes k}
\to \left(C^\bullet\right)^{\otimes k}\to 0.
\label{eq:edkernelk}
\end{equation}
Noting that $W^{\otimes k}\otimes\sheaf O$ is the degree $k$ component of
$(C^\bullet)^{\otimes k}$, the above discussion identifies 
the extension class of (\ref{eq:edkernelk}) with 
\[
\left(m_1\right)^{\otimes k}:W^{\otimes k}\to H^k(\X,\exterior k\sheaf E^\vee).
\]
Finally, there is a natural action of the symmetric group $S_k$
on the complex (\ref{eq:edkernelk})
obtained by permuting the factors and inserting signs to preserve the 
structure as a complex.  Then (\ref{eq:koszulType}) is simply the 
$S_k$-invariant subcomplex of (\ref{eq:edkernelk}), and the extension class
is $m_k$, the symmetrized version of $m_1^k$, as claimed.
\end{proof}

Each short exact sequence in \eqref{eq:endSequences} induces
coboundary maps
\begin{equation}
   \delta_j \!:\! H^j(\X, S_j^{(k)}) \lto H^{j+1}(\X, S_{j+1}^{(k)}),
   \label{eq:inducedCoboundary}
\end{equation}
which may be composed to a linear map 
\begin{equation}
\delta_{k-1} \circ \delta_{k-2} \cdots \circ \delta_0: H^0(\X, \Sym^k
   W\otimes \sheaf O_\X) \to  H^{k}(\X, \exterior k \sheaf E^\vee).
   \label{eq:inducedCoboundaryComposition}
\end{equation}
By first identifying $H^0(\X, \Sym^k W \otimes \sheaf O_X)$ with $\Sym^k
W$ and then applying the above composition of maps, we obtain  
precisely the linear map in \eqref{eq:kMultiplication}.  That is, 
\begin{equation}
m_k = \delta_{k-1} \circ \delta_{k-2} \circ \cdots \circ \delta_0.
\label{eq:mk}
\end{equation}

\begin{proposition}
   For all $1 \leq k \leq n$, $m_k$ is surjective. 
\label{prop:mk}
\end{proposition}

\begin{proof}
   Fix such a $k$;  
   for all $\ell$,
   $H^\ell(\X, \Z k \ell)=0$
   by Proposition \ref{prop:vanishing}. In the 
   long exact sequence in cohomology induced by the exact sequence
   \eqref{eq:endSequences}, we obtain the following subsequences for each
   $0 < \ell \leq k$:
   \begin{equation}
      \cdots \lto H^{\ell-1}(\X, \Z k \ell) \lto H^{\ell-1}(\X, \Ss k {\ell-1}) \lto H^\ell(\X, \Ss k \ell) \lto 0 
      \label{eq:Stuff}
   \end{equation}
   Thus, both the coboundary maps \eqref{eq:inducedCoboundary}
   and their composition \eqref{eq:inducedCoboundaryComposition} are surjective.
\end{proof}

\begin{remark}
In fact, most of the coboundary maps are isomorphisms; $H^{\ell-1}(\X,
\Z k \ell)$ is non-vanishing iff there is a primitive collection of
order $\ell$ in $\Sigma$, by Proposition~\ref{prop:vanishing}.  
This observation allows us to characterize
the kernel of the multiplication map: to find the relations amongst
the generators.
\label{rem:identifykernel}
\end{remark}

\subsection{Relations}
\label{sec:relations}

For each primitive collection $K$ we will exhibit an explicit
element $Q_K\in\ker m_k\subset \mathrm{Sym}^kW$, where $k=|K|$ as before.
We will see later that $\poly$ is the quotient of $\mathrm{Sym}W$ by the
ideal generated by the $Q_K$, with $K$ varying over all primitive collections.
For each $K$, we set 
\[
Z_K=\oplus_{\rho\in K}\sheaf O(-D_\rho),
\]
and for each $\ell\le k$ we set
\[
\ZK \ell = \exterior \ell Z_K\otimes\mathrm{Sym}^{k-\ell}W.
\]
Then the $\ZK \ell$ are the terms of a subcomplex of 
(\ref{eq:koszulType}).  Note that
\[
\ZK k\simeq\sheaf O\left(-D_K\right),
\]
and so $H^{k-1}(X, \ZK k)$ 
is isomorphic to $\CC$ by Proposition~\ref{prop:vanishing}.  

The following Lemma gives a procedure for computing $Q_K$.

\begin{lemma}
\ 
\vspace*{-1em}

\begin{enumerate}[i)]
\item Given $z_k\in Z^{k-1}(\ZK k)$, then one can simultaneously choose
$z_\ell\in C^{\ell-1}(\ZK \ell)$ so that
$\delta (z_\ell) = \alpha_{\ell+1}(z_{\ell+1})$ for each $1\le\ell\le k-1$,
where the $\alpha_i$ are given by (\ref{eq:koszulMaps}).
\item Given choices of the $z_\ell$ as above, then
$\alpha_1(z_1)\in Z^0(\mathrm{Sym}^kW\otimes\sheaf O)\simeq 
\mathrm{Sym}^kW$ depends only on the cohomology class $[z_k]\in
H^{k-1}(\ZK k)$ of $z_k$ 
and not on the choice of representative $z_k$ of that cohomology class or
any of the choices made for $z_\ell$ above.
\item  $\alpha_1(z_1)\in\ker (m_k)$.
\end{enumerate}
\label{lem:chasediagram}
\end{lemma}

\begin{proof}
For $i)$, it is enough to show that the $z_\ell$ can be successively chosen in
descending order to satisfy the required property.  
If $z_\ell$ has been chosen with $\ell>1$, then 
\[
\delta(\alpha_\ell (z_\ell))=\alpha_\ell(\delta (z_\ell))=\alpha_\ell
(\alpha_{\ell+1} (z_\ell)) =0,
\]
so that $\alpha_\ell (z_\ell)\in Z^{\ell-1}(\ZK {\ell-1})$ is a 
cocycle.  Since
$H^{\ell-1}(\ZK {\ell-1})=0$ by Proposition~\ref{prop:vanishing}, 
it follows that $\alpha_\ell (z_\ell)$
is a coboundary so can be written as $\delta (z_{\ell-1})$ for some 
$z_{\ell-1}\in C^{\ell-2}(\ZK {\ell-1})$.

For the first part of $ii)$, note that $\alpha_1(z_1)\in Z^0(\mathrm{Sym}^kW\otimes\sheaf O)$ is clear by
$\delta(\alpha_1(z_1))=\alpha_1(\delta(z_1))=\alpha_1(\alpha_2(z_2))=0$.  

Next, we prepare for an induction on $\ell$ by claiming that if 
$z_{\ell+1},\ldots,z_k$ are fixed and choices are only allowed to be made in 
$z_1,\ldots,z_\ell$, then $\alpha_1(z_1)$ is independent of the allowed 
choices.  The assertion of $ii)$ is simply this claim for $\ell=k$.

Fix $\ell<k$ and suppose that $z_r$ has been chosen for all
$r>\ell$ and we want to choose $z_\ell$ so that $\delta (z_\ell) =
\alpha_\ell(z_{\ell+1})$.  Once any $z_\ell$ is chosen, then any other
choice differs from $z_\ell$ by  
a cocycle in
$Z^{\ell-1}(\ZK \ell)$, which is necessarily a
coboundary $\delta y_\ell$ for some $(\ell-2)$-cochain $y_\ell\in C^{\ell-2}
(\ZK\ell)$, since $H^{\ell-1}(\ZK\ell)=0$.  
If $\ell=k$, a separate argument is required, but the conclusion is
the same: the only other choice for $z_k$ is to modify it by the addition
of a coboundary $\delta y_k$.

We start the induction with $\ell=1$.  Since $Z^0(\ZK 1)
=H^0(\ZK 1)=0$, there are no nontrivial cocyles, so that $z_1$
is unique and the independence is trivially true for $\ell=1$.  
Now suppose that the claim is true for some $\ell<k$ 
and we show that it is true for $\ell+1$.  As noted above, we can only
change the choice of $z_{\ell+1}$ to
$z_{\ell+1}+\delta(y_{\ell+1})$ and then we 
see what that does to the rest of the
computation.  Then $\alpha_{\ell+1}(z_{\ell+1})$ is replaced by
\[
\alpha_{\ell+1}(z_{\ell+1}+\delta(y_{\ell+1}))=\alpha_{\ell+1}(z_{\ell+1})+
\delta(\alpha_{\ell+1}(y_{\ell+1}))=\delta(z_\ell)+
\delta(\alpha_{\ell+1}(y_{\ell+1})).
\]
Thus, we can continue the computation by replacing $z_\ell$ by $z_\ell+
\alpha_{\ell+1}(y_{\ell+1})$.  Other choices for modifying $z_\ell$ are 
possible, but by the inductive hypothesis, other choices
won't affect $\alpha_1(z_1)$.
At the next step,
$\alpha_\ell(z_\ell)$ gets replaced by 
\[
\alpha_\ell(z_\ell+
\alpha_{\ell+1}(y_{\ell+1}))=\alpha_\ell(z_\ell)+\alpha_\ell
(\alpha_{\ell+1}(y_{\ell+1}))=\alpha_\ell(z_\ell),
\]
and no change in $z_{\ell-1}$ is required. The inductive hypothesis
takes care of the rest.

For $iii)$, let $i_{k}:H^{k-1}(\Z kk)\to H^{k-1}(S_{k-1}^{(k)})$
be the map in (\ref{eq:Stuff}).  Then 
we have constructed $\alpha_1(z_1)$ so that
\[
\left(\delta_{k-2}\circ\cdots\circ \delta_0\right)\left(\alpha(z_1)\right)=
i_{k}\left([z_k]\right).
\]
By the exactness of (\ref{eq:Stuff}), the image of $i_{k}$ is the kernel of
$\delta_{k-1}$.  So
\[
m_k\left(\alpha_1(z_1)\right) = \left(\delta_{k-1} \circ 
\cdots \circ \delta_0\right)
\left(\alpha(z_1)\right)=0
\]
as claimed.
\end{proof}

Note that Lemma~\ref{lem:chasediagram} gives rise to a well-defined map
\begin{equation}
\ell_K:H^{k-1}\big(\sheaf O(-D_K)\big)
\to \mathrm{Sym}^kW,\qquad
\ell_k([z_k])=\alpha_1(z_1).
\label{eq:lifttosw}
\end{equation}

We now examine the form of the image of $\ell_K$.  Let $z\in
H^{k-1}(\ZK k)$, and write $K=\{\rho_1,\ldots,\rho_k\}$.

\noindent{\em Claim:} $\ell_K(z)$ has the form
\begin{equation}
   \sum_{m_1 + m_2 + \cdots + m_k = 0} a_{\rho_1 m_1}
   \cdots a_{\rho_k m_k} \otimes p_{m_1m_2\cdots m_k},
\label{eq:formofqk}
\end{equation}
   where the $p_{m_1\cdots m_k}$ lie in $\Sym^k W$ and as before each
$m_i\in\Delta_{D_{\rho_i}} \cap M$.

To see this, we follow the definition of $\ell_K$, choosing the $z_\ell$
in decreasing order. After we have found 
$\alpha_\ell (z_\ell)$ and then make a choice
for $z_{\ell-1}$, the cocycle $\alpha_{\ell-1}(z_{\ell-1})$ depends linearly
on the coordinates $\{a_{\rho m}\}$ of the moduli space of deformations 
$\sheaf E$.  Since we compute $\alpha_\ell(z_\ell)$ for $\ell=1,\ldots,k$,
we conclude that $\ell_K(z)$ is homogeneous of degree $k$ as a function of
the $\{a_{\rho m}\}$, of the form (\ref{eq:formofqk}) without the qualifier
$m_1+\ldots+m_k=0$ in the sum.\footnote{In making this assertion, we are using
the fact the the coboundary maps are independent of bundle moduli, so that 
$z_\ell$ can recursively be chosen to have polynomial dependence on the 
$\{a_{\rho m}\}$, of degree $k-\ell$. It may only be possible to make choices
of $z_\ell$ on open subsets of bundle moduli; but since $\ell_K$ is 
well-defined we conclude that there is polynomial dependence globally in
bundle moduli.}  The restriction to $m_1+\ldots+m_k=0$ occurs because
$a_{\rho_1 m_1} \cdots a_{\rho_k m_k}$ only arises in combination with a
factor of $\chi^{m_1+\ldots+m_k}$, and
$\ell_K(z)\in H^0(\mathrm{Sym}^k W\otimes \sheaf O)$, which has pure
weight 0.  This proves the claim.

\begin{lemma}
Any primitive collection containing an edge $\rho$
necessarily contains the edge $\rho'$ if $D_{\rho'}$ is linearly 
equivalent to $D_\rho$.
\label{lem:primlin}
\end{lemma}

\begin{proof}
To see this, suppose $K=\{\rho_1,\ldots,\rho_k\}$ is a primitive collection.  
Following \cite{Batyrev:1993}, we have that
$v_{\rho_1}+\ldots+v_{\rho_k}$ lies in the relative interior of a unique 
cone
$\sigma\in\Sigma$, whose set of edges are disjoint from $K$.  
Letting the edges of
$\sigma$ be generated by primitive vectors
$w_1,\ldots,w_l$ we then have an identity
\begin{equation}
v_{\rho_1}\: + \: \ldots \: + \: v_{\rho_k} \: = \: \sum_{i=1}^l c_i w_i
\label{eq:primrel}
\end{equation}
with all $c_i>0$.

Suppose there is an edge $\rho'$ such that $D_{\rho'}$
is linearly equivalent to one of the $D_{\rho_i}$.  Then there exists
an $m\in M$ such that
\[
\langle m,v_{\rho'}\rangle = 1, \: \: \:
\langle m,v_{\rho_i}\rangle = -1, 
\: \: \:
\langle m,v'\rangle = 0
\]
where $v'$ is any edge generator distinct from $v_{\rho_i}$ or $v_{\rho'}$.
Applying $\langle m,\cdot\rangle$ to (\ref{eq:primrel}) we obtain
\begin{equation}
\label{mrel}
\sum_{j=1}^k\langle m,v_{\rho_j}\rangle
\: = \: \sum_{i=1}^l c_i \langle m,w_i\rangle.
\end{equation}
The only negative term in this equation is $\langle m,v_{\rho_i}\rangle$ 
on the left
hand side.  Therefore the left hand side must also contain a strictly
positive term. This can only happen if $v$ is one of the $v_{\rho_j}$.
\end{proof}

For later use, we note that the same argument shows that the $\{w_i\}$ are
closed under the relation
of linear equivalence of the corresponding divisors, and
in fact linearly equivalent terms have to arise with identical coefficients
$c_i$ in (\ref{eq:primrel}) as
the only way to get the identity $0=c_i-c_j$. 

We let $[K]$ denote the set of equivalence classes of edges in $K$
induced by linear equivalence of the corresponding divisors.
Similarly, we let $[K^-]$ denote the set of equivalence classes of the
edges spanned by the $w_i$ induced by linear equivalence of the
corresponding divisors.

We now discuss the linear terms of $E$ in more detail.
For each $\rho$, write 
\[
E_\rho=\sum_{m\in\Delta_{D_\rho}\cap M}
a_{\rho m}\; x_\rho\prod_{\rho'} x_{\rho'}^{\langle
m,v_\rho'\rangle}.
\]
The linear monomials occurring in $E_\rho$ come in two forms: 
$x_\rho$, and $x_{\rho'}$ for certain $\rho'\ne\rho$.  In the first case,
the monomial $x_\rho$ corresponds to $m = 0 \in \Delta_{D_\rho}$.
In the second case, $x_{\rho'}$ for $\rho'\ne \rho$ can be a term in
$E_\rho$ if for some $m \in M$ we have
\[
\langle m,v_{\rho'}\rangle =1, \: \: \:
\langle m,v_{\rho}\rangle  =-1, 
\text{ and }
\langle m,v_{\rho''}\rangle =0
\]
for all $\rho''\ne\rho,\rho'$.

Note the linear monomials $x_{\rho'}$ correspond to linear equivalences
$D_{\rho'}\sim D_\rho$.  The associated
$m$ is the one corresponding to the character whose divisor satisfies
\[
(\chi^m)=D_{\rho'}-D_\rho.
\]
Now partition the divisors $D_\rho$ into linear equivalence classes,
inducing a corresponding equivalence relation on $\Sigma(1)$:
the linear part of $E_\rho$ becomes
\begin{equation}
E_\rho^{\mathrm{lin}}=\sum_{\rho'\in[\rho]}a_{\rho\rho'}x^{\rho'},
\label{eq:elin}
\end{equation}
where $[\rho]$ is the equivalence class of $\rho$ under the equivalence 
relation just described.

For each such equivalence class $c=[\rho]$, we define a $|c|\times|c|$ matrix
$A_c$ with entries in $W$, where $|c|$ is the size of the equivalence class,
i.e.\ the number of divisors $D_{\rho'}$ linearly equivalent to $D_\rho$
(including $D_\rho$ itself).  The rows and columns of $A_c$ can be naturally
identified with the edges comprising the equivalence class.  The 
$(\rho,\rho')$ entry of $A_c$ is the coefficient $a_{\rho\rho'}$ in
(\ref{eq:elin}) above.

\begin{definition}
The matrix $A_c$  is the \emph{linear coefficient matrix} associated
to the divisor class corresponding to $c$.
   \label{def:linearCoefficients}
\end{definition}

\begin{remark}
   Note that this matrix is precisely the one denoted $A_{(\alpha)}^a$
   in equation 2.5 of \cite{McOrist:2007kp}.
\end{remark}

\begin{remark}
For $\sheaf E=TX$ with its standard Euler sequence, 
$A_c$ is diagonal with diagonal terms $[D_\rho]$.
\end{remark}

\begin{lemma}
   The image of $\ell_K$ only depends on 
   the linear terms of $E$.
   \label{lem:nonlinearIndependent}
\end{lemma}

\begin{proof}
We show that each term of (\ref{eq:formofqk}) depends only on the linear terms.
Since
   each $m_j$ is in $\Delta_{D_{\rho_j}}$, we know that
   $\langle m_j, v_{\rho_j} \rangle \geq -1$ while 
   $\langle m_j, v_\rho \rangle \geq 0$ for  $\rho\ne \rho_j$.  
Then the vanishing of
$m_1+\ldots+m_k$ implies that $\sum_{i=1}^k \langle m_i, v_{\rho_\ell}\rangle = 0$ for
   all $\ell$, so we are left with two possibilities for each $j$:

   \begin{enumerate}
      \item 
        $m_j = 0$
      \item 
         $m_j \ne 0$, and there exists a unique $i$ with $i \ne j$ such
that
         \[
         \langle m_j, v_{\rho_\ell} \rangle = 
         \begin{cases}
            1 & \ell = i\\
            -1 & \ell = j \\
            0 & \text{otherwise}.
         \end{cases}
         \]
   \end{enumerate}
As we have seen above, these $m_i$ each are associated with linear terms.
\end{proof}

We can at last explicitly describe the image of $\ell_K$.

\begin{lemma}
   Let $K$ be a primitive collection and $[K] = \{ [D_\rho] | \rho \in
   K\}$ as before.  Up to $\CC$-multiple, the image of $\ell_K$ is
   \[
Q_K:=   \prod_{c \in [K]} \det A_c,
   \]
   \label{lem:collationImage}
where the product is over the linear equivalence classes contained in $K$.
\end{lemma}

\begin{proof}
  Our application of Theorem~\ref{thm:simp} to compute
  \[H^{k-1}(X,\sheaf O(-D_K))\simeq\widetilde{H}^{k-2}(S^{k-2})\simeq\CC\]
  gives us more than the computation of the dimension of the
  cohomology group. If we choose an orientation of $S^{k-2}$ (determined by
  an ordering of the edges in $K$), the fundamental class of $S^{k-2}$
  gives an almost
  canonical generator $\gamma\in H^{k-1}(X,\sheaf O(-D_K))$ 
depending only on this orientation.

If we now choose any two edges $\rho$ and $\rho'$ in the same
equivalence class and interchange them in the ordering of $K$, 
this changes the orientation of
   $S^{k-2}$, and hence the sign of $\gamma$, while switching $E_\rho$
   and $E_{\rho'}$ and leaving everything else about the input data unchanged. 
Then $\ell_K(\gamma)$ is changed to $-\ell_K(\gamma)$.
Thus, given an ordered primitive collection,
 $\ell_K(\gamma)\in\mathrm{Sym}^kW$ is a degree $k$ polynomial function of
the $a_{\rho m}$,
   completely antisymmetric within each equivalence class $c$ of edges
corresponding to linear equivalence of divisors. We conclude
that up to multiple it is
   necessarily the product of determinants associated with blocks of
   linearly equivalent $D_i$.
\end{proof}

Since the expression $\det A_c$ will arise frequently in our computations,
we write
\[
Q_c:= \det A_c.
\]
For a perfect analogy with (\ref{eq:toriccohomology}), we would have to 
express the polymology as a quotient of the
homogeneous coordinate ring.
Instead, we content ourselves with a description of the polymology as a 
quotient of $\mathrm{Sym}^*W$.  We define 
the polymological analogue of the Stanley-Reisner ideal in $\Sym^*W$

\begin{definition}
The \emph{Stanley-Reisner ideal
   of $\sheaf E$} is 
   \begin{equation}
      SR(\X, \sheaf E) = 
      \left ( Q_K
      \;|\; K \text{ a
      primitive collection of } \Sigma\right) \subset \Sym^*W.
      \label{eq:polymologyStanleyReisner}
   \end{equation}
   \label{def:polymologyStanleyReisner}
\end{definition}

\begin{remark}
If $\sheaf E=TX$, then $Q_K=\prod_{\rho\in K}[D_\rho]$, and $SR(\X,T_X)$ is the
image of the usual Stanley-Reisner ideal $SR(\X)$ in $\mathrm{Sym}^*W$ 
under the quotient (\ref{eq:swlineq}).
\end{remark}

We now have the main theorem of this section.

\begin{theorem}
   Let $\X$ be a smooth projective toric variety, $W = \Pic(\X)
   \otimes_\ZZ \CC$, and $\sheaf E \to \X$ 
   a toric deformation of the tangent bundle of $\X$. 
   Then the polymology of $\sheaf E$ is isomorphic as a graded algebra to
   \begin{equation}
      \poly \simeq { \Sym^*W}/{SR(\X, \sheaf E)}
      \label{eq:polymologyIsomorphism}
   \end{equation}
\label{thm:polymologyIsomorphism}
\end{theorem}

\begin{proof}
The map
$m_k:\Sym^*W\to\poly$ is surjective by Proposition~\ref{prop:mk}.  We only have
to note that its kernel is $SR(\X, \sheaf E)$, by (\ref{eq:Stuff}), Remark
\ref{rem:identifykernel}, and Lemma \ref{lem:collationImage}.
\end{proof}

\section{Quantum sheaf cohomology}
\label{sec:quantum}

In the following, we denote the Mori cone of $\X$ by $\mori \X \subset
H_2(\X, \RR)$, and its integral points by $\moriz \X = \mori \X
\cap H_2(\X, \ZZ)$.  Since for a smooth toric variety
the Mori cone is generated by the curves associated to the cones in
$\Sigma(n-1)$ (see e.g. \cite{Cox:2010tv}), $\mori \X$ is polyhedral.

\subsection{Moduli space}

For each $\beta \in \moriz\X$, we describe the GLSM moduli space
variety $X_\beta$ after \cite{Witten:1993yc,Morrison:1994fr}.  We will
think of $X_\beta$ as a 
compactification of the space of holomorphic maps $\PP^1\to X$ of
class $\beta$, although in the GLSM the $\beta$ arise 
instead as a natural index for
the topological type of the gauge bundle on $\PP^1$.  

Fix a $\beta \in \moriz \X$, and let $d_\rho^\beta=
D_\rho\cdot\beta$.  An actual map $f:\PP^1\to \X$ can be described in 
homogeneous coordinates  as
\begin{equation}
x_\rho=f_\rho,\qquad f_\rho\in H^0(\PP^1,
\sheaf O(d_\rho^\beta)).\footnote{In the GLSM, the  $f_\rho$ are identified with 
the zero modes of certain charged chiral fields in the theory.}
\label{eq:glsmmaps}
\end{equation}
Accordingly, put
\[
\CC_\beta =
\bigoplus_{\rho} H^0(\PP^1, \sheaf O_{\PP^1}(d_\rho^\beta)).  
\]
There is a natural action of $G=\Hom(\Pic(\X), \CC^*)$ on
$\CC_\beta$, where $G$ acts on 
$f_\rho$ as multiplication by $g([D_\rho])$.
Thinking of a point in $\CC_\beta$
as a map $\CC^2 \to \CC^{\Sigma(1)}$, we
define a set $Z_\beta \subset \CC_\beta$ to be those tuples of sections
defining a map whose image is contained in $Z(\Sigma) \subset
\CC^{\Sigma(1)}$.  One can easily check that 
\begin{equation}
X_\beta = \left({\CC_\beta - Z_\beta}\right)/{G}
\label{eq:xbquotient}
\end{equation}
is a toric variety.\footnote{In the GLSM, the constraint leading to the
avoidance of $Z_\beta$ arises from FI terms.}   

We can alternatively describe $X_\beta$ as the  toric variety 
$X_{\Sigma_\beta}$
associated  to a fan $\Sigma_\beta$.
Let $(t_0,t_1)$ be homogeneous
coordinates on $\PP^1$ and for each $\rho$ with $d_\rho^\beta\ge0$ write
\begin{equation}
f_\rho=\sum_{i=0}^{d_\rho^\beta} f_{\rho_i}t_0^it_1^{d_\rho^\beta-i}.
\label{eq:fexp}
\end{equation}

Let $\CC_\beta^*\subset \CC_\beta$ be the subset on which all $f_{\rho_i}$
are nonzero.  Then $T_\beta=\CC_\beta^*/G\subset\X_\beta$ is a dense torus
acting on $\X_\beta$, giving it the structure of a toric variety.  Defining 
the lattice of 1-parameter subgroups of $T_\beta$ as $N_\beta$, we have a fan 
$\Sigma_\beta$ in $N_\beta\otimes\RR$ whose edges $\Sigma_\beta(1)$ naturally
correspond to $T_\beta$-invariant divisors $D_{\rho_i}$ defined by 
$f_{\rho_i}=0$.

We would like to identify the edges in $\Sigma_\beta(1)$ as 
$\rho_{i}$ with $\rho\in\Sigma(1)$ and $i=0,\ldots,
d_\rho^\beta$.  Then the $f_{\rho_i}$ would 
be naturally identified with the homogeneous coordinates of $X_\beta$.  

There is however a subtlety in that not all of the divisors $D_{\rho_i}$ 
defined by $f_{\rho_i}=0$ correspond to edges in the fan $\Sigma_\beta$, 
since it can happen that $D_{\rho_i}$ can
be empty.  Before explaining how these arise in general, an example will
clarify the phenomenon.

We consider
the Hirzebruch surface $F_n$ described as a 2-dimensional toric variety 
whose edges $\rho_1\ldots,\rho_4$ are respectively spanned by the 
vectors
\[
v_1=(1,0),\ v_2=(-1,n),\ v_3=(0,1),\ v_4=(0,-1).
\] 
Related aspects of this example are also discussed in \cite{Donagi:2011}.

Let $\beta=D_{\rho_3}$, the class of the $-n$ curve.  Then the
$d^\beta_{\rho_i}$ are respectively $1,1,-n,0$ for $i=1,\ldots,4$, and so
\[
\CC_\beta=
H^0(\PP^1, \sheaf O_{\PP^1}(1))^2\oplus 
H^0(\PP^1, \sheaf O_{\PP^1}).
\]
Note that $f_{\rho_3}\in H^0(\PP^1, \sheaf O_{\PP^1}(-n))$ is identically
zero, so that 
$Z_\beta$ contains the hyperplane $f_{{(\rho_4)}_0}=0$ owing to the primitive
collection $K=\{\rho_3,\rho_4\}$ in the fan $\Sigma$ of $F_n$.  It follows 
that $D_{(\rho_4)_0}$ is empty, hence $(\rho_4)_0$ is not in the fan
$\Sigma_\beta$.  In fact, $\Sigma_\beta$ is readily identified with the
standard toric fan for $\PP^3$,
with the four edges $(\rho_i)_j$ for $i=1,2$ and $j=0,1$.  We have
$\Sigma_\beta(1)=\{(\rho_i)_j\mid\ i=1,2, {\rm\ and\ } j=0,1\}$.

We now explain when $D_{\rho_i}$ can be empty.  
If $d^\beta_\rho<0$, then there are
no $f_{\rho_i}$.  If $d^\beta_\rho>0$, then
$f_{\rho_i}$ can be zero 
while $f_\rho$ from (\ref{eq:fexp}) is not identically zero, so $D_{\rho_i}$
is certainly nonempty.  It follows that $D_{\rho_i}$ is empty precisely
when $d^\beta_\rho=0$ and for some primitive
collection $K$ of $\Sigma$ containing $\rho$ we have
$d^{\beta}_{\rho'}<0$ for all $\rho'\ne\rho$ in $K$.  

We will consider these $\rho_0$ as 
\emph{degenerate edges} of $\Sigma_\beta$.

To formalize these considerations and facilitate a uniform treatment, 
we introduce the notation

\begin{equation}
\widehat{\Sigma}_\beta(1)=
\left\{\rho_i\mid\rho\in\Sigma(1),\ 0\le i\le d_\rho^\beta
\right\}.
\label{eq:enhancededges}
\end{equation}

In (\ref{eq:enhancededges}), $\widehat{\Sigma}_\beta(1)$ 
is a set of formal 
symbols. We can and will
identify those $\rho_i$ whose associated divisor $D_{\rho_i}$ is nonempty
with edges $\rho_i$ 
of the fan
$\Sigma_\beta(1)$.  No confusion should result from this slight abuse of 
notation.  Thus, $\widehat{\Sigma}_\beta(1)$ is an enhancement of the set of
edges of the fan $\Sigma_\beta$ by the degenerate edges.

In the case of $F_n$, we have 
$\widehat{\Sigma}_\beta(1)=\Sigma_\beta(1)\cup\{(\rho_4)_0\}$.

With this understanding, we can now alternatively 
specify the fan $\Sigma_\beta$
by specifying the primitive collections within 
the set $\widehat{\Sigma}_\beta(1)$
of enhanced edges.

Recalling  that $Z(\Sigma)$ is the union of the linear subspaces $L_K$,
we infer that $Z_\beta$ is the union over all $K$ of the subspaces of sections 
defining maps whose images are contained in $L_K\subset\CC^{\Sigma(1)}$.  
If $K=\{\rho_1,\ldots,\rho_k\}$, then this subspace is defined by imposing
$f_{{\rho_{j}}_i}=0$ for $1\le j\le k$.

Accordingly, for each primitive collection $K \subset \Sigma(1)$, define
$K_\beta \subset \widehat{\Sigma}_\beta(1)$ as the set of
all edges corresponding to those in $K$:
\[
K_\beta = \left\{ \left\{\rho_i \in \widehat{\Sigma}_\beta(1)\right\} \;|\; 
\rho \in K \right\}.
\]
It is straightforward to see that the $K_\beta$ are the primitive collections
for $\X_\beta$, with a natural extension of the notion of primitive 
collection to 
$\widehat{\Sigma}_\beta(1)$.  This is essentially because we can recognize
$Z_\beta$ as the union of the linear subspaces $L_{K_\beta}$.
Furthermore, the fan $\Sigma_\beta$ consists of the cones spanned
by all collections of edges in $\Sigma_\beta(1)$ that do not contain
any primitive collection $K_\beta$.  These general constructions of
toric geometry identify $Z_\beta$ with $Z(\Sigma_\beta)$.  Note that
if $\rho_0$ is a degenerate edge, then by this definition, the
singleton set $\{\rho_0\}$ is a primitive collection.  But this is
exactly what we want, since $\rho_0$ is not part of the actual fan.

For later use, we introduce the following numerical function.

\begin{definition}
Define the function $h^0:\ZZ\to \ZZ$ by
\[
h^0(x)=
\begin{cases}
   x+1 & x\ge -1\\
   0 & x \le -1,
\end{cases}
\]
or more concisely, $h^0(x) = h^0(\PP^1,\sheaf O_{\PP^1}(x))=\mathrm{max}(0,x+1)$.
\label{def:h0def}
\end{definition}

Note that for each edge $\rho \in
\Sigma(1)$ there are $h^0(d_\rho^\beta)$ edges in $\Sigma_\beta(1)$. 

\begin{proposition}
   For all $\beta \in \moriz \X$, $X_\beta$ is a smooth projective toric variety.
   \label{prop:moduliSmooth}
\end{proposition}

\begin{proof}
   See \cite[Example 7.2.3]{ciocan:2010tq} for a proof that $X_\beta$ is 
projective, and we have already explained that it is toric.  For smoothness,
it is elementary to show that a simplicial 
toric variety $X$ is smooth if and only if, upon writing 
$X=(\CC^{\Sigma(1)}-
Z(\Sigma))/G$, the action of $G$ on
$\CC^{\Sigma(1)}-Z(\Sigma)$ is free.  Thus, the smoothness of $X$ 
implies that $G$ acts freely on $\CC^{\Sigma(1)}-Z(\Sigma)$.

This easily implies that
the action of $G$ on $\CC^{\Sigma_\beta(1)}-Z(\Sigma_\beta)$
is also free as follows.  Let
$f\in \CC^{\Sigma_\beta(1)}-Z(\Sigma_\beta)$ with $f$ expressed in terms of the
$f_{\rho_i}$ as in (\ref{eq:fexp}).  Then for any primitive collection
$K$ and edge $\rho\in K$, we can't have
$f_{\rho_i}=0$ for all $i$ by the definition of $K_\beta$. 
It follows that the $f_{\rho_i}(t_0,t_1)$ can't
all vanish for generic $(t_0,t_1)\in\CC^2$, hence
there exists $(t_0,t_1)\in\CC^2$ with
$f(t_0,t_1)\not\in Z$, where $f(t_0,t_1)\in \CC^{\Sigma(1)}$ has 
coordinates $f_{\rho}(t_0,t_1)$.  
If in addition such an $f$ is a fixed point, 
then $f(t_0,t_1)\in \CC^{\Sigma(1)}-Z(\Sigma)$ would
also be a fixed point, contradicting the smoothness of $X$.  
Therefore $X_\beta$
is smooth as well.
\end{proof}

\begin{remark}
   $\M 0 = \X$; the moduli space of constant maps is $\X$ itself.
\end{remark}

\subsection{Induced sheaf}

Since $\M \beta$ is a smooth projective toric variety, we have its Euler exact
sequence
\begin{equation}
0 \to \mathcal O_{\M \beta}\otimes_{\ZZ} \Pic(X_\beta)^\vee 
\overset {F_0} \longrightarrow
\bigoplus_{\rho_i\in \Sigma_\beta(1)} \mathcal O_{\M \beta}(D_{ {\rho_i}}) \to
T_{\M \beta} \to 0,
\label{eq:xbtes}
\end{equation}
where the $\rho_i^\text{th}$ component of the morphism $F_0$ is $f_{ \rho_i} 
\otimes [D_{ \rho_i}]$.  

Recalling that $D_{\rho_0}$ is empty for a degenerate edge $\rho_0$,
by adding trivial line bundles to each of the first two nonzero bundles in 
(\ref{eq:xbtes}) for each degenerate edge,
we obtain a modification of the Euler sequence of $X_\beta$
\begin{equation}
0 \to \mathcal O_{\M \beta}\otimes W^\vee \overset {\hat{F}_0} \longrightarrow
\bigoplus_{\rho_i\in \widehat{\Sigma}_\beta(1)} \mathcal O_{\M \beta}(D_{ {\rho_i}}) \to
T_{\M \beta} \to 0,
\label{eq:modtes}
\end{equation}
where the $\rho_i^\text{th}$ component of the morphism $\hat{F}_0$ is 
$f_{ \rho_i} 
\otimes [D_{ \rho}]$. 

We add a few words of clarification on the relationship between 
(\ref{eq:xbtes}) and (\ref{eq:modtes}), even though it is not essential for
the sequel.

Let 
\[
\widehat{\CC}_\beta =
\widehat{\bigoplus}_{\rho} H^0(\PP^1, \sheaf O_{\PP^1}(d_\rho^\beta)),
\]
where the hat over the direct sum means that we are omitting the components
of $\CC_\beta$ associated with degenerate edges.  Let $G'\subset G$ be the
subgroup which acts as the identity on 
the linear subspace of $\CC_\beta$ defined
by the vanishing of the $f_{\rho_i}$ with $\rho_i\in \Sigma_\beta(1)$ 
(rather than $\widehat{\Sigma}_\beta(1)$).  Then we have
\begin{equation}
X_\beta=\left(\widehat{\CC}_\beta-\left(Z_\beta\cap \widehat{\CC}_\beta\right)
\right)/G',
\label{eq:usualquotient}
\end{equation}
where $\widehat{\CC}_\beta$ is viewed as a subspace of $\CC_\beta$ in the 
natural way.  In fact, (\ref{eq:usualquotient}) is just the usual description
of $X_\beta$ as a quotient constructed from the fan $\Sigma_\beta(1)$.

By (\ref{eq:picg}), $\Pic(\X)$ is canonically isomorphic to $\Hom(G,\CC^*)$
and $\Pic(\X_\beta)$ is canonically isomorphic to $\Hom(G',\CC^*)$.  
The inclusion $G'\hookrightarrow G$ therefore induces a mapping
$\Pic(\X)\to\Pic(\X_\beta)$ which is needed in justifying the claimed 
relationship between (\ref{eq:xbtes}) and (\ref{eq:modtes}).
It is the modified version (\ref{eq:modtes}) 
of the toric euler sequence
that gets deformed by a deformation $\sheaf E$ of $TX$.  
With slight abuse
of notation, we will refer to deformations of the map $\hat{F}_0$ as giving
rise to toric deformations of the tangent bundle of $X_\beta$.  No confusion
should result.

We now define the induced sheaf $\sheaf E_\beta$ precisely
as dictated by the GLSM.
On $\X$, we defined a bundle $\sheaf E$ in terms of 
sections $E_\rho$ of $\sheaf O_\X(D_\rho)\otimes W$, for each 
$\rho\in\Sigma(1)$. 
We shall now associate to these sections corresponding sections of 
$\sheaf O_{\M \beta}(D_{\rho_i})\otimes W$, 
leading to a toric deformation
$\sheaf E_\beta$ of the tangent bundle of $\X_\beta$ in the modified
sense just explained.

We express each $E_\rho$ in terms of homogeneous coordinates; for emphasis
we write this as $E_\rho=E_\rho(x)$.  Over 
$\X_\beta$ we then make the substitutions
$x_\rho=f_\rho(t_0,t_1)$ for each $\rho\in \Sigma(1)$ to obtain expressions
that we abbreviate as $E_\rho(f(t))$.  We now collect powers of $t_0$ and 
$t_1$, writing the result as
\begin{equation}
E_\rho(f(t))=\sum_{i=0}^{d_\rho^\beta} E_{\rho_i}(f)t_0^it_1^{d_\rho^\beta-i}.
\label{eq:inducedBundle}
\end{equation}
The $W$-valued expressions $E_{\rho_i}(f)$ are then 
interpreted as the components of a toric deformation $\sheaf E_\beta$
of the tangent bundle of $\X_\beta$ in the modified sense, defined by 
the exact sequence

\begin{equation}
0 \to \mathcal O_{\M \beta}\otimes W^\vee \overset {\hat{F}} \longrightarrow
\bigoplus_{\rho_i\in \widehat{\Sigma}_\beta(1)} \mathcal O_{\M \beta}(D_{ {\rho_i}}) \to
\sheaf E_\beta \to 0,
\label{eq:defmodtes}
\end{equation}
where the components of $\hat{F}$ are as described above.

\begin{proposition}
   $\sheaf E_\beta$ is locally-free if $\sheaf E$ is.
\end{proposition}

\begin{proof}
The local freeness of $\sheaf E$ is equivalent to the assertion that the
$E_\rho(x)$ span $W$ for all $x\in \CC^{\Sigma(1)}-Z(\Sigma)$---this is just
the surjectivity of (\ref{eq:elocfree}).  

Now let $f\in C_\beta-Z_\beta$. 
Taking a generic $t=(t_0,t_1)\in \CC^2$
as in the proof of Proposition~\ref{prop:moduliSmooth}, it 
follows that $f(t_0,t_1)\in \CC^{\Sigma(1)}-Z(\Sigma)$, hence the 
$E_\rho(f(t))$ span $W$.  But (\ref{eq:inducedBundle}) says that $E_\rho(f(t))$
is in the span of the $E_{\rho_i}(f)$. It follows immediately that 
the $E_{\rho_i}(f)$ span $W$ as well, hence $E_\beta$ is locally
free.
\end{proof}

We now turn to the computation of the polymology of $(\X_\beta,
\sheaf E_\beta)$.  Since $\sheaf E_\beta$ is a toric deformation of the
tangent bundle of $\X_\beta$ in the modified sense, its polymology algebra 
can be computed by the same method as in Section~\ref{sec:polymology}, 
resulting in 
a description of its polymology as a quotient of the symmetric algebra of $W$.
The only change that is needed is to consider degenerate edges.  Recall 
that for a degenerate edge
$\rho_0$, $D_{\rho_0}$ is empty.   So we have to supplement
Proposition~\ref{prop:vanishing} with
\[
H^\ell(\X_\beta,\sheaf O(-D_{\rho_0}))=\left\{
\begin{array}{cl}
\CC&\ell=0\\
0 & \ell >0
\end{array}
\right. 
\]
for degenerate edges $\rho_0$.
The reasoning in Section~\ref{sec:polymology} produces an element
$Q_{[\rho_0]}\in W$ associated to the primitive collection $\{\rho_0\}$
which is in the kernel of the map $W\to H^1(\X_\beta,\sheaf E_\beta^\vee)$
arising from the long exact sequence associated to the dual
of (\ref{eq:defmodtes}).  We let 
$\widehat{SR}(X_\beta,\sheaf E_\beta)$ be the ideal in $\Sym^*W$ generated by
$SR(X_\beta,\sheaf E_\beta)$ and the 
$Q_{[\rho_0]}$ just described.

To compute $\widehat{SR}(X_\beta,
\sheaf E_\beta)$, we just need to compute the $Q_{K_\beta}$.

Fix a $\rho\in\Sigma(1)$ and write
\[
E_\rho=\sum_{\rho}'A_{\rho\rho'}x_{\rho'}+\cdots,
\]
where the sum is over edges $\rho'$ with $D_{\rho'}$ linearly equivalent
to $D_\rho$ and the $\cdots$ represent the omitted nonlinear terms.  Said 
differently, $A$ is the linear coefficient matrix $A_c$ associated to
the linear equivalence class $c$ containing $\rho$.  Then
\[
\begin{array}{ccl}
E_\rho(f)& = &\sum_{\rho'}A_{\rho\rho'}f_{\rho'}+\cdots\\
&=&\sum_{\rho',i}A_{\rho\rho'}f_{\rho'_i}t_0^it_1^{d_\beta-i}+\cdots,
\end{array}
\]
so that
\begin{equation}
E_{\rho_i}=\sum_{\rho'}A_{\rho\rho'}f_{\rho'_i}+\cdots.
\label{eq:block}
\end{equation}
Denoting the analogue of $A_c$ for $\sheaf E_\beta$ by $(A_\beta)_c$,
equation (\ref{eq:block}) says that $(A_\beta)_c$ has a block diagonal form
consisting of $h^0(d_\beta)$ copies of $A_c$.  Thus
\[
\det\left(A_\beta\right)_c=Q_c^{h^0(d_\beta)}.
\]
It follows immediately that
\begin{equation}
Q_{K_\beta}=\prod_{c\in [K]} Q_c^{h^0(D_c\cdot\beta)}
\label{eq:qbk}
\end{equation}
and 
\[
H^*_{\sheaf E_\beta}=\Sym^*W/\widehat{SR}(X_\beta,\sheaf E_\beta),
\]
where
\[
\widehat{SR}(X_\beta,\sheaf E_\beta)=      \left ( Q_{K_\beta}
      \;|\; K \text{ a
      primitive collection of } \Sigma\right) \subset \Sym^*W.
\]

Note that if $K$ is a primitive collection of $\Sigma$ containing an edge
$\rho$ of $\Sigma$ with $\rho_0$ a degenerate edge of $\Sigma_\beta$, then 
from (\ref{eq:qbk}) we get $Q_{K_\beta}=Q_{[\rho]}$, where $[\rho]=\{\rho\}$
is the linear equivalence class of $\rho$.  But this is precisely the linear
part of $\widehat{SR}(X_\beta,\sheaf E_\beta)$ by our earlier discussion.

As with the classical case, we will define the correlation functions in the
sector labelled by $\beta$ as 
elements of the one-dimensional vector space
$H^{n_\beta}(\X_\beta,\exterior {n_\beta}\sheaf E_\beta)$, where $n_\beta$ is the dimension of
$\X_\beta$.  The precise definition will be spelled out later, but first
we have to grapple with the normalization issue.

The correlation functions will be obtained as usual by adding the contributions
over the sectors $\beta$.  However, since these contributions live in {\em different\/}
one-dimensional vector spaces, we will describe a distingushed space $H^*$,
together with a collection of isomorphisms 
\[
H^{n_\beta}_{\sheaf E_\beta} \simeq H^*
\]
for each $\beta$, in which the contributions will be summed.  This will be accomplished by assembling the $H^{n_\beta}_{\sheaf E_\beta}$
into a direct system in the next section.

\subsection{Direct system of polymologies}
\label{sec:direct}

For every $\beta \in H_2(\X, \ZZ)$, we have constructed an induced
deformation $\sheaf E_\beta$ of the toric Euler sequence on $\X_\beta$.  
The algebra $\polybeta$ is generated by elements of 
$H^1(\X_\beta,\sheaf E_\beta^\vee) \simeq W$.  

We now construct a direct system from these
polymologies and show that the one-dimensional spaces
$H^{n_\beta}(\X_\beta,\exterior {n_\beta}\sheaf E_\beta^\vee)$ are preserved by
the maps of the direct system, hence by restriction also give a direct system.
We will show that the direct limit is a one-dimensional vector space, in which
the correlation functions take their values.

For each $c$ corresponding to a linear equivalence class of divisors $D_\rho$,
we put $d_c^\beta=d_\rho^\beta$, for any $\rho$ in the equivalence class.

\begin{definition}
For classes $\beta,\beta'\in H_2(X,\ZZ)$, we say that $\beta'$ 
{\em dominates\/} $\beta$
if $\beta'-\beta$ is effective and
$h^0(d_c^{\beta'})\ge h^0(d_c^{\beta})$ for all
linear equivalence classes $c$ of the irreducible toric divisors $D_\rho$.
\end{definition}

If $\beta'$ dominates $\beta$, 
we define the expression
\begin{equation}
   \label{def:r}
   R_{\beta'\beta}=\prod_c Q_c^{h^0(d_c^{\beta'})-h^0(d_c^\beta)}\in 
\mathrm{Sym}^* W.
\end{equation}

\begin{lemma}
   Suppose that $\beta'$ dominates $\beta$.  Then
   \[
   R_{\beta'\beta}\left(
SR(\X_\beta,\sheaf E_\beta)\right)\subset SR(\X_{\beta'},\sheaf E_{\beta'}).
   \]
\label{lem:domsr}
\end{lemma}

\begin{proof}
We will show that for each primitive collection $K$,
$R_{\beta'\beta}Q_{K_\beta}$
is a multiple of $Q_{K_{\beta'}}$.  This will suffice to prove the lemma,
by the definitions of $SR(\X_\beta,\sheaf E_\beta)$ and 
$SR(\X_{\beta'},\sheaf E_{\beta'})$.

For this, it suffices to compare the powers of $Q_c$ occurring 
in $R_{\beta'\beta}Q_{K_\beta}$ and $Q_{K_{\beta'}}$, for each $c$.
If $c\in [K]$, then
the exponent of $Q_c$ in $R_{\beta'\beta}Q_{K_\beta}$ is
\[
h^0(d_c^{\beta'})-h^0(d_c^\beta)+h^0(d_c^\beta) = h^0(d_c^{\beta'}),
\]
which is the exponent of $Q_c$ in $Q_{K_{\beta'}}$.

If $c\not\in[K]$, then the exponent of $Q_c$ in $R_{\beta'\beta}Q_{K_\beta}$ is
$h^0(d_c^{\beta'})-h^0(d_c^\beta)$,
the exponent of $Q_c$ in $Q_{K_{\beta'}}$ is $0$, and the required inequality
holds by the dominance assumption.
\end{proof}

Whenever $\beta'$ dominates $\beta$, 
multiplication by $R_{\beta'\beta}$ induces a well-defined map
\[
f_{\beta'\beta}:H^*_{\sheaf E_\beta}(X_\beta)\to
H^*_{\sheaf E_{\beta'}}(X_{\beta'})
\]
by Lemma~\ref{lem:domsr}.  It is straightforward to verify that the
$\{f_{\beta'\beta}\}$ form a direct system.  We have to show that the
maps are compatible and that any $\beta_1$ and $\beta_2$ are dominated by
some $\beta$.  Compatibility is
obvious, and given any $\beta_1$ and $\beta_2$,
choose a $\beta_3$ effective that has positive 
intersection with each $D_c$\footnote{An intersection of ample divisors
suffices.} and
set $\beta=\beta_1+\beta_2+n\beta_3$
for some $n\gg 0$.
Then $\beta$ dominates both $\beta_1$
and $\beta_2$.

For simplicity of notation, let 
$H^{n_\beta}_{\sheaf E_\beta}$ be the degree $n_\beta$ part of 
$\Sym^*W/(SR(X_\beta),\sheaf E_\beta)$ (which is canonically isomorphic to
$H^{n_\beta}(X_\beta,\exterior {n_\beta}\sheaf E_\beta^\vee)$).

\begin{lemma}
   \ 
   \vspace*{-1em}
   \label{lem:direct}
\begin{enumerate}[i)]
\item If $\beta'$ dominates $\beta$, then $f_{\beta'\beta}(H^{n_\beta}_{\sheaf E_\beta})
\subset H^{n_{\beta'}}_{\sheaf E_{\beta'}}$.  Thus the maps
\[g_{\beta'\beta}:=f_{\beta'\beta}|_{H^{n_\beta}_{\sheaf E_\beta}}:
H^{n_\beta}\to 
H^{n_{\beta'}}_{\sheaf E_{\beta'}}\]
also form a direct system.
\item If $\X_\beta$ and $\X_{\beta'}$ are nonempty, then 
$g_{\beta'\beta}$ is an isomorphism 
for deformations $\sheaf E$ sufficiently close to $TX$ in the moduli space
of toric deformations of $TX$.
\end{enumerate}
\end{lemma}

It follows immediately from Lemma~\ref{lem:direct} that the direct limit
\[
H^*:=\lim_{\stackrel{\longrightarrow}{g}}H^{n_\beta}_{\sheaf E_\beta}
\]
is a 1-dimensional vector space,
and the induced maps $i_\beta:H^*_\beta\to H^*$ are isomorphisms.  
The correlation functions will all take values in $H^*$.  

\begin{proof}[Proof of Lemma \ref{lem:direct}]
The emptiness of $\X_\beta$ is equivalent to $D_c\cdot\beta<0$ for all $c$
that are part of some fixed primitive collection $K$, by the definition of the
primitive collections for $\X_\beta$.

For i), we just have to show that the cohomology degrees are compatible 
with $g_{\beta'\beta}$.  

Noting that $Q_c$ has degree $|c|$, the number of divisors $D_\rho$ in the
corresponding linear equivalence class, we see that
\[
\begin{array}{ccl}
\deg R_{\beta'\beta}&=&\sum_c |c|\left(h^0(d_c^{\beta'})-h^0(d_c^{\beta})\right)\\
&=&\sum_{\rho}\left(h^0(d_\rho^{\beta'})-
h^0(d_\rho^{\beta})\right).
\end{array}
\]
Thus we must show
\begin{equation}
n_{\beta'}=n_\beta+\sum_{\rho}\left(h^0(d_\rho^{\beta'})-h^0(d_\rho^{\beta})\right).
\label{eq:dimident}
\end{equation}
Before computing $n_\beta=\dim \X_\beta$ we note that for a general toric variety $X$ we have
\[
\dim \X=|\Sigma(1)|-h^2(\X)=\left(\sum_{\rho}1\right)-h^2(X),
\]
as follows from the quotient description (\ref{eq:quotientdescription}) and $\dim G = \mathrm{rank}(Pic(\X))
=h^2(\X)$.
Applying the same calculation to $\X_\beta$, we count the edges in $\Sigma_\beta(1)$ and recall that
$h^2(\X_\beta)=h^2(X)$ to conclude
\begin{equation}
n_\beta=\dim \X_\beta  
= \left(\sum_{\rho}h^0(d_\rho^\beta)\right)-h^2(X),
\label{eq:nb}
\end{equation}
and (\ref{eq:dimident}) follows immediately.

For ii), we just have to show that the map $g_{\beta'\beta}$ of one-dimensional vector spaces is nonzero.

{\em Claim:\/} If $R_{\beta'\beta}\not\in SR(\X_{\beta'},{\sheaf E_{\beta'}})$, 
then $g_{\beta'\beta}$ is an isomorphism.

To justify the claim, the hypotheses can be restated as saying that
$[R_{\beta'\beta}]$ is nonzero in $H^*_{\sheaf E_{\beta'}}$.  Since $\X_{\beta'}$ is smooth and $\sheaf E_{\beta'}$ is locally free,
$H^*_{\sheaf E_{\beta'}}$ 
satisfies
Poincar\'e duality.  Thus, there exists an element $p\in\mathrm{Sym}^* W$ such that
$[pR_{\beta'\beta}]$ is a nonzero element of $H^{n_{\beta'}}_{\sheaf E_{\beta'}}$.
Thus $g_{\beta'\beta}(p)\ne 0$, justifying the claim.

It therefore suffices to show that $R_{\beta'\beta}\not\in SR(\X_{\beta'},\sheaf E_{\beta'})$.
We do this by first verifying it for $\sheaf E = TX$.
Once we show that, we
have proven the first part of the lemma in a neighborhood of $TX$
by the closedness of the ideal membership condition.

We now assume that $\sheaf E = T_X$ and identify the polymology of $\sheaf E$ with the cohomology of $X$.
We will show that $g_{\beta'\beta}$ applied to the cohomology class of a point
of $\X_\beta$ is the cohomology class of a point
of $\X_{\beta'}$.\footnote{This matching of point classes appeared in the
context of ordinary cohomology in \cite{Morrison:1994fr}, and was part of
our motivation for the definition of $R_{\beta'\beta}$.}  Since the class of a point is nontrivial, it follows that
$g_{\beta'\beta}$ is nonzero.  
In this case, it is easy to see that $Q_c=[\prod_{\rho\in c} x_\rho]$ by looking at the components of
the toric Euler sequence (\ref{eq:Tpresentation}).

Recall that $S=\CC[x_\rho \; |\; \rho \in \Sigma(1)]$ is 
the homogeneous coordinate ring of $X$.  We have the 
ring homorphism $S\to \mathrm{Sym}^*W$ defined by 
taking $x_\rho$ to the cohomology class of $D_\rho$.  Starting with a polynomial $p$, we can take its image in $\mathrm{Sym}^* W$ and then take a further
quotient by SR$(\X)$
to get a cohomology class $[p]\in H^*(X)$. 
We also let $S_\beta$ be the homogeneous
coordinate ring of $\X_\beta$, with a similar map to $\mathrm{Sym}^*W$ and we use
the notation $[p]_\beta$ for the image of a polynomial in $H^*(\X_\beta)$.

If $\beta'$ dominates $\beta$, then
$S_\beta$ can be regarded as a subring of $S_{\beta'}$ in a natural way as
follows.  Recall that $S_\beta$ is generated by the variables 
$x_{\rho_1},\ldots,x_{\rho_{d^\beta_\rho}}$ while $S_{\beta'}$ is generated by the 
variables $x_{\rho_1},\ldots,x_{\rho_{d^{\beta'}_\rho}}$ .  The dominance 
assumption implies that $d^{\beta'}_\rho\ge d^{\beta}_\rho$ for each $\rho$,
so the shared nomenclature of the variables provides a natural 
embedding of $S_\beta$ in $S_{\beta'}$.
Furthermore $R_{\beta'\beta}$ is simply
the image in $\mathrm{Sym}W$ of the product 
\[
m=\prod_{\rho} \prod_{i=d^\beta_\rho+1}^{d^{\beta'}_\rho}x_{\rho_i}
\]
corresponding
to the additional edges 
added in going from $\X_\beta$ to $\X_{\beta'}$.  

When we multiply $R_{\beta'\beta}$ as represented by $m$ by the class of a 
point represented by a product of variables in $S_\beta$ 
corresponding to a maximal cone $\sigma$ of $\X_\beta$,
it is clear that multiplication by $m$ corresponds to 
simply appending the additional 
edges $\rho_{d^\beta_\rho+1},\ldots,\rho_{d^{\beta'}_\rho}$ needed to complete
$\sigma$ to a maximal cone of $\X_{\beta'}$.  Since the resulting 
monomial represents the class of a point, a nonzero cohomology class,
the resulting monomial in $S_{\beta'}$ cannot be in the Stanley-Reisner
ideal.

Part of the above argument appeared in \cite{Morrison:1994fr}.
\end{proof}

\subsection{Correlation Functions.}  In this section, we will define
the correlation functions.  In the half-twisted GLSM, correlation functions in
sector $\beta$ can be nonzero only if the operators have degree
$c_1(X)\cdot\beta +\dim(X)$.  Since
$c_1(X)=\sum_{\rho}D_\rho$, it follows that
\[
c_1(X)\cdot \beta=\sum_{\rho}d_\rho^\beta, 
\]
hence
\begin{equation}
c_1(X)\cdot\beta +\dim(X)=\sum_{\rho}\left(d_\rho^\beta+1\right)-h^2(X), 
\label{eq:virdim}
\end{equation}
The formula (\ref{eq:virdim}) plays the role of the virtual or expected dimension of Gromov-Witten theory.

Note that (\ref{eq:nb}) and (\ref{eq:virdim}) differ only in that 
$h^0(d_\rho^\beta)$ in
(\ref{eq:nb}) is replaced by $d_\rho^\beta+1$ in (\ref{eq:virdim}).  These are in fact equal, unless
$d_\rho^\beta\le -2$ for some $\rho$.  Such a situation is the analogue
of excess dimension in Gromov-Witten theory.  In our situation, we have both
the excess dimension of $\X_\beta$ and the excess rank of $\sheaf E_\beta$.

In this case, to compensate, we need something to play the role of the obstruction classes of
Gromov-Witten theory.  These were called four-fermi terms in \cite{Katz:2004nn} because of how
they arose in the path integral.  

Define \( h^1(x)=h^1(\sheaf O_{\PP^1}(x)) \)
in analogy with Definition~\ref{def:h0def}.
Then our formula for the four-fermi terms is
\begin{equation}
F_\beta=\prod_c Q_c^{h^1(d_c^\beta)},
\label{eq:fourfermi}
\end{equation}
where the product is over all linear equivalence classes $c$ of the 
divisors $D_\rho$.

At last, we can define the correlation functions in sector $\beta$.
Let
\[
p\in\Sym^{c_1(X)\cdot\beta+\dim X}W.
\]  
A simple computation of degree shows that $pF_\beta\in H^{n_\beta}_{\sheaf E_\beta}$:
the degree of $pF_\beta$ is 
\[
\begin{split}
   \text{deg}(p F_\beta) &= c_1(X)\cdot\beta+\dim X+\sum_c|c|h^1(d_c^\beta)\\
&=\sum_{\rho}\left(d_\rho^\beta+1\right)-h^2(X)+\sum_{\rho}h^1(d_\rho^\beta)\\ 
&=\sum_{\rho}h^0(d_\rho^\beta)-h^2(X) \\
&=n_\beta.
\end{split}
\]
In the third line of the above computation, we have used the identity
\begin{equation}
h^0(d_\rho^\beta)-h^1(d_\rho^\beta)=d_\rho^\beta+1,
\label{eq:rr}
\end{equation}
which is Riemann Roch for $\PP^1$.  

Finally, the correlation function is defined as
\begin{equation}
\langle p\rangle_\beta = i_\beta\left(p\, F_\beta\right)\in H^*,
\label{eq:cordef}
\end{equation}
where $i_\beta$ was defined immediately following the statement of 
Lemma~\ref{lem:direct}.
Following our discussion at the beginning of this section, if
$p\in \Sym^dW$ with $d\ne c_1(X)\cdot \beta+\dim X$, we define $\langle p\rangle_\beta=0$.
By design, all correlation functions live in the same one-dimensional
vector space $H^*$, so can be added over $\beta$.  To formalize the sum,
we recall one version of the Novikov ring.

\begin{definition}
The {\em Novikov ring\/} $\CC[ q^\beta]$ of $\X$ is the ring
generated over $\CC$ by the formal expressions
$q^\beta$ for each $\beta\in\mori \X$, subject to the relations 
$q^\beta q^{\beta'}=q^{\beta+\beta'}$.
\end{definition}

For any $p\in\Sym^*W$ we define the correlation function
\[
\langle p\rangle=\sum_\beta\langle p\rangle_\beta q^\beta\in\Sym^*W\otimes
\CC[ q^\beta].
\]

\subsection{Quantum Cohomology Ring}
In the quasi-topological sector of the 
half-twisted GLSM, as in all quantum field theories where correlation
functions are independent of the insertion point, an operator
$\sheaf O$ is the trivial operator iff
\[\langle \sheaf O,\sheaf O_1,\ldots,
\sheaf O_k\rangle=0
\] 
for all operators $\sheaf O_1,\ldots,
\sheaf O_k$.  Accordingly, the trivial operators with coefficients in the 
Novikov ring form an ideal in the
ring $\Sym^*W\otimes \CC[ q^\beta]$, which we suggestively call
the {\em quantum Stanley-Reisner ideal\/} and write as 
$QSR(X,\sheaf E)$.

\begin{definition}
The {\em quantum sheaf cohomology\/} of $(X,\sheaf E)$ is the ring
\[
QH^*_{\sheaf E}(X)=\left(\Sym^*W\otimes\CC[ q^\beta]\right)/
QSR(X,\sheaf E).
\]
\label{def:quantumsr}
\end{definition}

By definition, if we set all $q^\beta=0$, the correlation functions
become the classical correlation functions described in 
Section~\ref{sec:polymology} and the quantum Stanley-Reisner ideal
$QSR(X,\sheaf E)$ becomes the ordinary Stanley-Reisner ideal
$SR(X,\sheaf E)$.  Thus, $QSR(X,\sheaf E)$ is a deformation of
$SR(X,\sheaf E)$.  Accordingly, we expect $QSR(X,\sheaf E)$ to be generated
by deformations of the generators $Q_K$ of $SR(X,\sheaf E)$.  In fact,
passing to the localization to make the comparison, the relations in 
\cite{McOrist:2008ji} are of just 
this form.  We will view these as predictions for the generators in our 
set-up and then prove that they are correct.

Fixing a primitive collection $K$, we rewrite (\ref{eq:primrel}) as
\begin{equation}
\label{aprimrel}
\sum a_\rho v_\rho=0,
\end{equation}  
where $a_\rho=1$ for each $\rho\in K$.
Then there exists a unique $\beta_K\in H_2(X,\ZZ)$ such that
\[
d_\rho^{\beta_K}=D_\rho\cdot\beta_K=a_\rho\qquad\forall\rho.
\]
Furthermore, the $\beta_K$ generate the cone of effective curves.  Details of
these assertions can be found in \cite{Cox:2010tv}.

Then there is proposed a relation
\begin{equation}
\label{eq:qcrelation}
\prod_{c\in [K]}
Q_c
=q^{\beta_K}\prod_{c\in [K^-]}
Q_c^{-d_c^{\beta_K}}
\end{equation}
Since the left hand side of (\ref{eq:qcrelation}) is just $Q_K$,
(\ref{eq:qcrelation}) says that the quantities
\[
Q_K-q^{\beta_K}\prod_{c\in [K^-]}
Q_c^{-d_c^{\beta_K}}
\] 
are in $QSR(X,\sheaf E)$ and specialize to the 
generator $Q_K$ of $SR(X,\sheaf E)$ when the $q^\beta$ are set to 0.

By definition,
(\ref{eq:qcrelation}) is equivalent to the identity of correlation functions
\begin{equation}
\langle Y\prod_{c\in [K]}
Q_c\rangle_{\beta+
\beta_K}=
\langle Y \prod_{c\in [K^-]}
Q_c^{-d_c^{\beta_K}}\rangle_\beta
\label{eq:corrident}
\end{equation}
for any $Y\in \mathrm{Sym}^*W$ and $\beta\in H_2(X,\ZZ)$.  

\begin{theorem}
The quantum cohomology relations (\ref{eq:qcrelation})
hold for all primitive collections $K$.
\label{thm:qcrelations}
\end{theorem}

\begin{proof}
We show (\ref{eq:corrident}) for any $Y\in \Sym^*W$.
Choosing $\beta'$ dominating both $\beta$ and $\beta+\beta_K$, we have to
show the equality  
\begin{equation}
R_{\beta',\beta+\beta_K}F_{\beta+\beta_K}Y\prod_{c\in K}
Q_c=R_{\beta'\beta} F_\beta Y
\prod_{c\in [K^-]}
Q_c^{-D_c\cdot\beta_K}
\label{eq:qcrelcor}
\end{equation}
as elements of the quotient ring
\[
\Sym^*W/SR(\X_{\beta'},\sheaf E_{\beta'}).
\]
In fact, we will see that (\ref{eq:qcrelcor}) holds in $\Sym^*W$.
For this it suffices to show
\begin{equation}
R_{\beta',\beta+\beta_K}F_{\beta+\beta_K}\prod_{c\in K}
Q_c=R_{\beta'\beta} F_\beta 
\prod_{c\in [K^-]}
Q_c^{-D_c\cdot\beta_K}
\label{eq:corinbp}
\end{equation}

Both sides of (\ref{eq:corinbp}) expand to products of powers of the $Q_c$,
so we just have to check the exponents of each $Q_c$.  
We break this up into three cases,
according to whether $d_c^{\beta_K}$ is
positive, negative, or zero, or equivalently, $c\in [K],\ c\in [K^-]$,
or $c$ in neither $[K]$ nor $[K^-]$.  In any case, we note that
\begin{equation}
d_c^{\beta+\beta_K}=d_c^\beta+d_c^{\beta_K}
\label{eq:dclinear}
\end{equation}
If $d_c^{\beta_K}=0$, then the required equality of exponents is
\[
h^0(d_c^{\beta'})-h^0(d_c^{\beta+\beta_K})+h^1(d_c^{\beta+\beta_K})=
h^0(d_c^{\beta'})-h^0(d_c^{\beta})+h^1(d_c^{\beta}).
\]
However, in this case, 
$d_c^\beta=d_c^{\beta+\beta_K}$ by
(\ref{eq:dclinear}) and equality is clear.

If $d_c^{\beta_K}>0$ then
$c\in [K]$ and
the required equality of exponents is
\[
h^0(d_c^{\beta'})-h^0(d_c^{\beta+\beta_K})+
h^1(d_c^{\beta+\beta_K})+1=
h^0(d_c^{\beta'})-h^0(d_c^{\beta})+h^1(d_c^{\beta}).
\]
The equality follows immediately from (\ref{eq:dclinear}), $d_c^\beta=1$,
and two applications of (\ref{eq:rr}) (one time 
with $\beta$ replaced by $\beta+\beta_K$).

Finally, if $d_c^{\beta_K}<0$, then $c\in [K^-]$ and we have to show
\[
h^0(d_c^{\beta'})-h^0(d_c^{\beta+\beta_K})+h^1(d_c^{\beta+\beta_K})=
h^0(d_c^{\beta'})-h^0(d_c^{\beta})+h^1(d_c^{\beta})-d_c^{\beta_K},
\]
which is easily verified in the same way.
\end{proof}

\bigskip From physics, we expect this result to generalize to complete
intersections in toric varieties.  Suppose that $Y\subset X$ is a
complete intersection in a toric variety $X$.  Then the tangent bundle
of $Y$ is the cohomology of a monad.  The cohomology of a 
small deformation of this
monad will be a vector bundle $\sheaf E$ on $Y$.  Then $\mathrm{Sym}^*W$
generates a subalgebra of the polymology of $\sheaf E$, which we call 
the \emph{toric polymology} of $\sheaf E$.  We write
\begin{equation}
H^*_{\sheaf E}(X)^{\mathrm{toric}}=\left(\mathrm{Sym}^*W\right)/
\mathrm{SR}(X,\sheaf E),
\label{eq:srci}
\end{equation}
where for present purposes $\mathrm{SR}(X,\sheaf E)$ is defined by 
(\ref{eq:srci}).

\begin{conjecture}
There is a \emph{toric quantum sheaf cohomology ring} 
$QH^*_{\sheaf E}(X)^{\mathrm{toric}}$ which is of the form
\[
QH^*_{\sheaf E}(X)^{\mathrm{toric}}=\left(\mathrm{Sym}^*W\right)/
\mathrm{QSR}(X,\sheaf E),
\]
where $\mathrm{QSR}(X,\sheaf E)$ specializes to $\mathrm{SR}(X,\sheaf E)$
after setting all the $q^\beta$ to zero.
\end{conjecture}

Note that we have proven this conjecture for $Y=X$.  

\section*{Acknowledgements}

We would like to thank Matt Ballard, Jock McOrist, Ilarion Melnikov, and Hal Schenck
for helpful discussions.
RD acknowledges partial support by NSF
grants 0908487 and 0636606.
JG acknowledges support by NSF grant 0636606.
SK acknowledges partial support from NSF grant DMS-05-55678.
ES acknowledges partial support by NSF grants DMS-0705381 and PHY-0755614,
hospitality of the Erwin Schrodinger Institute in Vienna and the University
of Pennsylvania while part of the work was completed.

\newpage
\bibliography{qsctesd}

\end{document}